\documentclass[11pt]{article}
\usepackage{amsfonts}
\usepackage{amsmath,amssymb,amsthm,amsfonts,enumerate}
\usepackage{authblk}
\usepackage{graphicx}
\usepackage{float}
\usepackage[scriptsize,nooneline]{subfigure}
\usepackage{indentfirst}
\usepackage{cite}
\usepackage{color}
\usepackage{mathrsfs}
\usepackage{epstopdf}
\usepackage{setspace}
\usepackage[labelfont=footnotesize,font=footnotesize]{caption}
\usepackage[colorlinks, citecolor=red]{hyperref}

\newtheorem{theorem}{Theorem}[section]

\newtheorem{definition}{Definition}[section]
\newtheorem{proposition}{Proposition}[section]
\newtheorem{remark}{Remark}[section]
\newtheorem{corollary}{Corollary}[section]
\def\geq{\geqslant}\def\leq{\leqslant}

\begin{document}
	
\title{\bf Homogeneous mixed Herz-Morrey spaces and its Applications }
\author[1]{Mingwei Shi}
\author[2]{Jiang Zhou\thanks{Corresponding author. The research was supported by National Natural Science Foundation of China (12061069).}}

\affil[1,2]{College of Mathematics and System Sciences, Xinjiang University, Urumqi 830046, China}

\date{}
\maketitle
{\bf Abstract:}{
In this paper, we introduce homogeneous mixed Herz-Morrey spaces $M\dot{K}_{p,\vec{q}}^{\alpha,\lambda}(\mathbb{R}^n)$  and show it's some properties. Firstly,  the boundedness of sublinear operators, fractional type operators in homogeneous mixed Herz-Morrey spaces is investigated. In particular, the above results are still valid for Calder$\acute{o}$n-Zygmund operators and fractional maximal operators. Lastly, the boundedness of  their commutators  in  homogeneous mixed Herz-Morrey spaces is obtained. }\par


\baselineskip 15pt

\section{Introduction}
 In 1938,  Morrey \cite{CC1938} introduced Morrey spaces      $\mathcal{M}_{p,\lambda}(\mathbb{R}^n)$. Due to the importance of Morrey spaces   $\mathcal{M}_{p,\lambda}(\mathbb{R}^n)$ in partial differential equation, many authors defined the generalization of Morrey spaces and combined it with other spaces. In 2005, Lu and Xu \cite{SL2005} introduced  homogeneous Herz-Morrey spaces $M\dot{K}_{p,q}^{\alpha,\lambda}(\mathbb{R}^n)$  which are  generalization of Morrey spaces $\mathcal{M}_{p,\lambda}(\mathbb{R}^n)$  and homogeneous  Herz spaces $\dot{K}_{q}^{\alpha, p}(\mathbb{R}^n)$. And defined as follows ($\alpha \in \mathbb{R}$, $ 0<p\leq \infty$, $0\leq\lambda< \infty$, $0 < q< \infty$),
 $$
 M\dot{K}_{p,q}^{\alpha,\lambda}(\mathbb{R}^n):=\{f\in L_{loc}^{q}: \|f\|_{M\dot{K}_{p,q}^{\alpha,\lambda}(\mathbb{R}^n)}<\infty \},
 $$
 where
 $$
\|f\|_{M\dot{K}_{p,q}^{\alpha,\lambda}(\mathbb{R}^n)}= \mathop{sup}\limits_{k_0 \in \mathbb{Z}}2^{-k_{0}\lambda} \left\{\sum_{k=-\infty}^{k_0}2^{k \alpha p}\|f\chi_{k}\|_{L^{q}}^p\right\}^{\frac{1}{p}}<\infty,
 $$
it is obvious that $M \dot{K}_{p, q}^{\alpha, 0}\left(\mathbb{R}^{n}\right)=\dot{K}_{q}^{\alpha, p}\left(\mathbb{R}^{n}\right)$ and $M_{q}^{\lambda}\left(\mathbb{R}^{n}\right) \subset M \dot{K}_{q, q}^{0, \lambda}\left(\mathbb{R}^{n}\right)$.\par
 In recent years, researchers have done more research in homogeneous Herz-Morrey spaces   $M\dot{K}_{p,q}^{\alpha,\lambda}(\mathbb{R}^n)$. Specifically speaking in 2005, Lu and Xu \cite{SL2005} proved the boundedness of rough Singular integral operators. In 2009, Tao, Shi and Zhang \cite{TY009} discussed the boundedness of multilinear Riesz potential operators on  Herz-Morrey spaces  $M\dot{K}_{p,q}^{\alpha,\lambda}(\mathbb{R}^n)$. In 2010, Izuki \cite{M2010} researched fractional integral with variable exponent on Herz-Morrey spaces $M \dot{K}_{q, p(\cdot)}^{\alpha, \lambda}\left(\mathbb{R}^{n}\right)(\alpha \in \mathbb{R},   0<q<\infty, p(\cdot) \in \mathscr{P}\left(\mathbb{R}^{n}\right), 0 \leq \lambda<\infty)$; Gao and Jiang \cite{GW2010} proved the boundedness of a class of linear commutators; Shi, Tao and Zhang  \cite{YTZ2010} considered multilinear Riesz potential operators on Herz-Morrey Spaces with non-doubling measures. In 2012,  Zhou and Zhu  \cite{ZJ2012} testified the boundedness of multilinear singular integrals and their commutators on Herz-Morrey spaces with non-doubling measures. In 2015, Mizuta and Ohno \cite{YT2015} obtained  Herz-Morrey spaces $M \dot{K}_{q, p(\cdot)}^{\alpha, \lambda}\left(\mathbb{R}^{n}\right)$ of variable exponent, Riesz potential operators and duality.
\par
In 1961, Benedek and Panzone\cite{AB1961} introduced mixed Lebesgue spaces $L^{\vec{p}}(\mathbb{R}^n) (0<\vec{p}\leq \infty)$. Recently, due to the wider usefulness of mixed function spaces on partial differential equations. Mixed spaces have more applications and complex structures, so the study of mixed norm is more difficult in harmonic analysis. Results of previous studies,  mixed Morrey spaces $M_{\vec{q}}^{\lambda}(\mathbb{R}^n)(0\leq \lambda < \infty,0<\vec{q}\leq \infty)$ were first introduced by Nogayama\cite{TN2019} in 2019; homogeneous mixed Herz spaces $\dot{K}_{\vec{q}}^{\alpha, p}(\mathbb{R}^n)(\alpha \in \mathbb{R},0<p\leq \infty,0<\vec{q}\leq \infty)$ were first  introduced by Wei 
\cite {MW2021} in 2021.
\par
In this paper, mixed Herz-Morrey spaces are defined, and some properties and boundedness of operators are testified.
\par
 Below is the structure of this paper: in Section 2, we introduce homogeneous mixed  Herz-Morrey spaces $M\dot{K}_{p,\vec{q}}^{\alpha,\lambda}(\mathbb{R}^n)$ and study some of their properties. In Sections 3-4, the boundedness of a wide class of sublinear operators, fractional type operators and their commutators on homogeneous mixed Herz-Morrey spaces $M\dot{K}_{p,\vec{q}}^{\alpha,\lambda}(\mathbb{R}^n)$ are proved.
\section{Definitions and Preliminaries }
In this section, some definitions and properties about  homogeneous mixed Herz-Morrey spaces $M\dot{K}_{p,\vec{q}}^{\alpha,\lambda}(\mathbb{R}^n)$ are introduced. There are some symbols as follows. Let $B_k=\{x\in\mathbb{R}^n: |x|\leq 2^k\}$ and $A_k=B_k\backslash{B_{k-1}}$ for any $k\in \mathbb{Z}$. Denote $\chi_{_k}=\chi_{_{A_k}}$, for any $k\in \mathbb{Z}$,  where $\chi_{_E}$ is the characteristic function of set $E$. And throughout this paper, we use following notations. The letter $\vec{q}$ will denote n-tuples of the numbers in $(0,\infty)^n$ $(n \geq 1)$, $\vec{q}=\left(q_1, q_2, \dots, q_n\right)$. By definition, the inequality $0<\vec{q}<\infty$ means that $0<q_i<\infty$ for all $i$. For $\vec{q}=\left(q_1, q_2, \dots, q_n\right)$ and $k\in \mathbb{R}$, let\\
$$\frac{1}{\vec{q}}=\left(\frac{1}{q_1}, \frac{1}{q_2},\dots, \frac{1}{q_n}\right),\quad \frac{\vec{q}}{k}=\left(\frac{q_1}{k}, \frac{q_2}{k},\dots, \frac{q_n}{k}\right),\quad \vec{q^{\prime}}=\left(q_1^{\prime},q_2^{\prime},\dots, q_n^{\prime}\right),$$
where $q_j^{\prime}=\frac{q_j}{q_j-1}$ is the conjugate exponent of $q_j$. $|B|$ denotes the volume of the ball $B$.

\begin{definition}(Mixed Lebesgue spaces) \cite{AB1961}
	Let $\vec{p}=\left(p_1, p_2,\dots, p_n\right)\in (0,\infty]^n.$ Then  mixed Lebesgue spaces $L^{\vec{p}}(\mathbb{R}^n)$ is defined to be the set of all measurable functions $f$ such that their quasi-norms\\
	$$\left\|f\right \|_{\vec{p}}\equiv \left(\int_{\mathbb{R}}\dots\left(\int_{\mathbb{R}}\left(\int_{\mathbb{R}}\mid f(x_1, x_2,\dots, x_n)\mid^{p_1}dx_1\right)^{\frac{p_2}{p_1}}dx_2\right)^{\frac{p_3}{p_2}}\dots dx_n\right)^{\frac{1}{p_n}}< \infty ,$$
	if $p_j=\infty,$ then we have to make appropriate modifications.
	
\end{definition}

\begin{definition}(Mixed Morrey spaces) \cite{TN2019}
	Let $0\leq \lambda < \infty$,  $0<\vec{q}\leq \infty$, where $\vec{q}=(q_1,q_2,\dots,q_n)$, then define mixed morrey spaces $M_{\vec{q}}^{\lambda}(\mathbb{R}^n)$ by\\
	$$M_{\vec{q}}^{\lambda}(\mathbb{R}^n)(\mathbb{R}^n):=\{f\in L_{loc}^{\vec{q}}: \|f\|_{M_{\vec{q}}^{\lambda}(\mathbb{R}^n)}< \infty\},$$\\
	where$$\|f\|_{M_{\vec{q}}^{\lambda}(\mathbb{R}^n)}=sup ~r^{-\lambda}\left\| f\chi_{k} \right\|_{L^{\vec{q}}}.$$	
\end{definition}

\begin{definition}(Homogeneous  Mixed Herz spaces) \cite{MW2021}
Let $\alpha \in \mathbb{R}$, $0<p\leq \infty$, $0<\vec{q}\leq \infty$, where $\vec{q}=(q_1,q_2,\dots,q_n)$, then define  homogeneous mixed Herz spaces $\dot{K}_{\vec{q}}^{\alpha, p}(\mathbb{R}^n)$  by\\
$$\dot{K}_{\vec{q}}^{\alpha, p}(\mathbb{R}^n):=\{f\in L_{loc}^{\vec{q}}: \|f\|_{\dot{K}_{\vec{q}}^{\alpha, p}}< \infty\},$$\\
where$$\|f\|_{\dot{K}_{\vec{q}}^{\alpha, p}}=\left(\sum_{k\in \mathbb{Z}}2^{k \alpha p}\|f\chi_k\|_{L^{\vec{q}}}^p\right)^{\frac{1}{p}}.$$	
\end{definition}
Generalize the above $M_{\vec{q}}^{\lambda}(\mathbb{R}^n)$ and $\dot{K}_{\vec{q}}^{\alpha, p}(\mathbb{R}^n)$ to a new class of function spaces called homogeneous mixed Herz-Morrey spaces  $M\dot{K}_{p,\vec{q}}^{\alpha,\lambda}(\mathbb{R}^n)$, which is defined as follows.
\begin{definition} (Homogeneous Mixed Herz-Morrey spaces)	
Let $\alpha \in \mathbb{R}$, $ 0<p< \infty$, $0\leq\lambda< \infty$, $1 \leq \vec{q}< \infty$, where $\vec{q}=(q_1,q_2,\dots,q_n)$, then define homogeneous mixed Herz-Morrey spaces $M\dot{K}_{p,\vec{q}}^{\alpha,\lambda}(\mathbb{R}^n)$  by\\
	$$M\dot{K}_{p,\vec{q}}^{\alpha,\lambda}(\mathbb{R}^n):=\{f\in L_{loc}^{\vec{q}}: \|f\|_{M\dot{K}_{p,\vec{q}}^{\alpha,\lambda}(\mathbb{R}^n)}<\infty \},$$\\
	where
	$$\|f\|_{M\dot{K}_{p,\vec{q}}^{\alpha,\lambda}}=\|f\|_{M\dot{K}_{p,\vec{q}}^{\alpha,\lambda}(\mathbb{R}^n)}= \mathop{sup}\limits_{k_0 \in \mathbb{Z}}2^{-k_{0}\lambda} \left\{\sum_{k=-\infty}^{k_0}2^{k \alpha p}\|f\chi_{k}\|_{L^{\vec{q}}}^p\right\}^{\frac{1}{p}}.$$	
\end{definition}
\begin{remark}
	As well known that $L^{\vec{p}}(\mathbb{R}^n)=L^p(\mathbb{R}^n)$ for $\vec{p}=\left(p, p,\dots, p\right)$, so it is obvious that $M\dot{K}_{p,\vec{q}}^{\alpha,\lambda}(\mathbb{R}^n)=M\dot{K}_{p,q}^{\alpha,\lambda}(\mathbb{R}^n)$ for $\vec{q}=\left(q, q,\dots, q\right)$ and  $ M\dot{K}_{p,\vec{q}}^{\alpha,0}(\mathbb{R}^n)=\dot{K}_{\vec{q}}^{\alpha, p}(\mathbb{R}^n).$	
\end{remark}
\begin{proposition}
	Homogeneous mixed Herz-Morrey spaces $M\dot{K}_{p,\vec{q}}^{\alpha,\lambda}(\mathbb{R}^n)$  are  quasi-Banach spaces if $  0<p, \vec{q} <1$, $M\dot{K}_{p,\vec{q}}^{\alpha,\lambda}(\mathbb{R}^n)$  are Banach spaces if $1\leq p,\vec{q}\leq \infty$.	
\end{proposition}
\begin{proof}
	First, it is proved that $\|\cdot\|_{ M\dot{K}_{p,\vec{q}}^{\alpha,\lambda}(\mathbb{R}^n)}$ in homogeneous mixed Herz-Morrey spaces    $M\dot{K}_{p,\vec{q}}^{\alpha,\lambda}(\mathbb{R}^n)$. The positivity and the homogeneity are clear. Therefore, the most important thing is to check the quasi-triangle inequality. Without losing the generality, it is assumed that $1\leq p$.
$$
\begin{aligned}
\|f+g\|_{M\dot{K}^{\alpha,\lambda}_{p,\vec{q}}}
&= \mathop{sup}\limits_{k_0\in \mathbb{Z}}2^{-k_{0}\lambda} \left\{\sum_{k=-\infty}^{k_0}2^{k \alpha p}\|(f+g)\chi_{k}\|_{L^{\vec{q}}}^p\right\}^{\frac{1}{p}}\\
&\leq \mathop{sup}\limits_{k_0\in \mathbb{Z}}2^{-k_{0}\lambda} \left\{\sum_{k=-\infty}^{k_0}2^{k \alpha p}\|f\chi_{k}+g\chi_{k}\|_{L^{\vec{q}}}^p\right\}^{\frac{1}{p}}\\
&\leq  \mathop{sup}\limits_{k_0\in \mathbb{Z}}2^{-k_{0}\lambda}\left(\sum_{k=-\infty}^{k_{0}}2^{k \alpha p}\left(max(1,2^{\sum\limits_
	{i=1}^n\frac{1-q_i}{qi}})\left(\|f\chi_k\|_{L^{\vec{q}}}+\|g\chi_k\|_{L^{\vec{q}}}\right)\right)^p\right)^{\frac{1}{p}}	\\
&\leq  max(1,2^{\sum\limits_{i=1}^n\frac{1-q_i}{qi}})\mathop{sup}\limits_{k_0\in \mathbb{Z}}2^{-k_{0}\lambda}\left(\sum_{k=-\infty}^{k_{0}}2^{k \alpha p} 2^{p-1} \left(\|f\chi_k\|_{L^{\vec{q}}}^p+\|g\chi_k\|_{L^{\vec{q}}}^p\right)\right)^{\frac{1}{p}}	\\
&\leq  max(1,2^{\sum\limits_{i=1}^n\frac{1-q_i}{qi}})  2^{1-p}2^{p-1}
\mathop{sup}\limits_{k_0\in \mathbb{Z}}2^{-k_{0}\lambda} \left[\left(\sum_{k=-\infty}^{k_{0}}2^{k \alpha p} \|f\chi_k\|_{L^{\vec{q}}}^p\right)^{\frac{1}{p}}+
\left(\sum_{k=-\infty}^{k_{0}}2^{k \alpha p} \|g\chi_k\|_{L^{\vec{q}}}^p\right)^{\frac{1}{p}}	\right]\\
			\end{aligned}
$$
$\quad~\quad \quad \quad\quad\quad 
\leq max(1,2^{\sum\limits_{i=1}^n\frac{1-q_i}{qi}}) (\|f\|_{M\dot{K}^{\alpha,\lambda}_{p,\vec{q}}}+\|g\|_{M\dot{K}^{\alpha,\lambda}_{p,\vec{q}}}).
$\\	
Then, when $1\leq q_i \leq \infty$ for any $i=1,2,\dots,n$, $\|\cdot\|_{ M\dot{K}_{p,\vec{q}}^{\alpha,\lambda}(\mathbb{R}^n)}$ is obtained. When $0<q_i<1$ for exists $i=1,2,\dots,n$, it is a quasi-norm. Next, we will display the completeness of space $ M\dot{K}_{p,\vec{q}}^{\alpha,\lambda}(\mathbb{R}^n).~	$ 
 Let a Cauchy sequence $\{f_j\}_{j=1}^{\infty}$ satisfied that
$$\|f_{j+1}-f_{j}\|_{M\dot{K}_{p,\vec{q}}^{\alpha,\lambda}} \leq 2^{-j}.$$ 
Let $f(x)=f_1(x)+\sum\limits_{j=1}^{\infty}\left(f_{j+1}(x)-f_{j}(x)\right)=\lim\limits_{j\rightarrow \infty}f_j(x)$, then $f\in M\dot{K}_{p,\vec{q}}^{\alpha,\lambda}(\mathbb{R}^n)$, \\
$$
\begin{aligned}
\|f\|_{M\dot{K}_{p,\vec{q}}^{\alpha,\lambda}}
&=	\left\|f_1+\sum_{j=1}^{\infty}\left(f_{j+1}-f_{j}\right)\right\|_{M\dot{K}_{p,\vec{q}}^{\alpha,\lambda}}\\
&\leq \|f_1\|_{M\dot{K}_{p,\vec{q}}^{\alpha,\lambda}}+\left\|\sum_{j=1}^{\infty}\left(f_{j+1}-f_{j}\right)\right\|_{M\dot{K}_{p,\vec{q}}^{\alpha,\lambda}}\\
&\leq \|f_1\|_{M\dot{K}_{p,\vec{q}}^{\alpha,\lambda}}+\sum_{j=1}^{\infty}\|f_{j+1}-f_{j}\|_{M\dot{K}_{p,\vec{q}}^{\alpha,\lambda}}\\
&\leq C\left(\|f_1\|_{M\dot{K}_{p,\vec{q}}^{\alpha,\lambda}}+\sum_{j=1}^{\infty}2^{-j}\right) \leq C.
\end{aligned}
$$
  Furthermore, $\{f_j\}_{j=1}^{\infty}$ converge to $f\in M\dot{K}_{p,\vec{q}}^{\alpha,\lambda}(\mathbb{R}^n)$,
  $$
\begin{aligned}
\|f-f_J\|_{M\dot{K}_{p,\vec{q}}^{\alpha,\lambda}}
&=\left\|f_1+\sum_{j=1}^{\infty}\left(f_{j+1}-f_{j}\right)-f_J\right\|_{M\dot{K}_{p,\vec{q}}^{\alpha,\lambda}}\\
&\leq \left\|\sum_{j=1}^{\infty}\left(f_{j+1}-f_{j}\right)-f_J-\sum_{j=1}^{J-1}\left(f_{j+1}-f_{j}\right)-f_J\right\|_{M\dot{K}_{p,\vec{q}}^{\alpha,\lambda}}\\
&\leq \left\|\sum_{j=J+1}^{\infty}\left(f_{j+1}-f_{j}\right)\right\|_{M\dot{K}_{p,\vec{q}}^{\alpha,\lambda}}\\
&\leq 2^{1-J}.
\end{aligned}
$$
Then
$$\lim\limits_{J\rightarrow \infty}\|f-f_J\|_{M\dot{K}_{p,\vec{q}}^{\alpha,\lambda}}=0.$$
Thus,  $M\dot{K}_{p,\vec{q}}^{\alpha,\lambda}(\mathbb{R}^n)$ are quasi-Banach spaces.
\end{proof}

\begin{proposition}
	Let $\alpha, \alpha_1, \alpha_2, \dots, \alpha_n \in \mathbb{R}$ and $\alpha=\alpha_1+\alpha_2+ \dots +\alpha_n$,   $0<p< \infty$, $0\leq\lambda< \infty$, $1 \leq \vec{q}< \infty$, where $\vec{q}=(q_1,q_2,\dots,q_n)$, then
	$$M\dot{K}_{p,\vec{q}}^{\alpha,\lambda}=
	M^{\lambda}_{\vec{q}~(\mathbb{R}^n, |x_1|^{\alpha_1q_1}\times|x_2|^{\alpha_2q_2}\dots \times|x_n|^{\alpha_n q_n})}.$$	
\end{proposition}

\begin{proof}
	\begin{align*}
	\|f\|_{M\dot{K}_{p,\vec{q}}^{\alpha,\lambda}}
&=	\mathop{sup}\limits_{k_0\in \mathbb{Z}}2^{-k_{0}\lambda} \left(\sum_{k=-\infty}^{k_0}2^{k\alpha q_n}\left(\int_{\mathbb{R}}\dots \left(\int_{\mathbb{R}}\left(\int_{\mathbb{R}}|f\chi_k(x)|^{q_1}dx_1\right)^{\frac{q_2}{q_1}}dx_2\right)^{\frac{q_3}{q_2}}\dots dx_n\right)^{\frac{q_n}{q_n}}\right)^{\frac{1}{q_n}}	\\
	&\leq 
	\mathop{sup}\limits_{k_0\in \mathbb{Z}}2^{-k_{0}\lambda}\left(\sum_{k=-\infty}^{k_0}2^{k\alpha q_n}\left(\int_{2^{k-1}\leq|x_n|<2^k}\dots \left(\int_{2^{k-1}\leq|x_1|<2^k}|f(x)|^{q_1}dx_1\right)^{\frac{q_2}{q_1}}\dots dx_n\right)^{\frac{q_n}{q_n}}\right)^{\frac{1}{q_n}}	\\
  &=\mathop{sup}\limits_{k_0\in \mathbb{Z}}2^{-k_{0}\lambda}\left(\sum_{k=-\infty}^{k_0}\left(\int_{2^{k-1}\leq|x_n|<2^k}\dots \left(\int_{2^{k-1}\leq|x_1|<2^k}|f(x)|^{q_1}2^{k\alpha_1 q_1}dx_1\right)^{\frac{q_2}{q_1}}\dots 2^{k\alpha_n q_n}dx_n\right)\right)^{\frac{1}{q_n}}	\\
	&\sim
	\mathop{sup}\limits_{k_0\in \mathbb{Z}}2^{-k_{0}\lambda}\left(\sum_{k=-\infty}^{k_0}\left(\int_{2^{k-1}\leq|x_n|<2^k}\dots \left(\int_{2^{k-1}\leq|x_1|<2^k}|f(x)|^{q_1}|x_1|^{\alpha_1 q_1}dx_1\right)^{\frac{q_2}{q_1}}\dots |x_n|^{\alpha_n q_n}dx_n\right)\right)^{\frac{1}{q_n}}	\\
	&=
	\mathop{sup}\limits_{k_0\in \mathbb{Z}}2^{-k_{0}\lambda}\left(\left(\int_{\mathbb{R}}\dots \left(\int_{\mathbb{R}}|f(x)|^{q_1}|x_1|^{\alpha_1 q_1}dx_1\right)^{\frac{q_2}{q_1}}\dots |x_n|^{\alpha_n q_n}dx_n\right)\right)^{\frac{1}{q_n}}	\\
	&=M^{\lambda}_{\vec{q}~(\mathbb{R}^n, |x_1|^{\alpha_1q_1}\times|x_2|^{\alpha_2q_2}\dots \times|x_n|^{\alpha_n q_n})}.
	\end{align*}	
\end{proof}
 
\begin{proposition}
	Let $\alpha \in \mathbb{R}$,  $0<p< \infty$, $0\leq\lambda< \infty$, $1 \leq \vec{q}< \infty$, where $\vec{q}=(q_1,q_2,\dots,q_n)$. The following inclusions are valid.\\
	$(1)$ If $p_1\leq p_2$, then $M\dot{K}_{p_1,\vec{q}}^{\alpha,\lambda}(\mathbb{R}^n) \subset  M\dot{K}_{p_2,\vec{q}}^{\alpha,\lambda}(\mathbb{R}^n) $.\\
	$(2)$ If $\alpha_2\leq \alpha_1$, then $M\dot{K}_{p,\vec{q}}^{\alpha_1,\lambda}(\mathbb{R}^n) \subset  M\dot{K}_{p,\vec{q}}^{\alpha_2,\lambda}(\mathbb{R}^n)$.\\
	$(3)$ If $\vec{q_{_1}}\leq \vec{q_{_2}}$, then~$ M\dot{K}_{p,\vec{q}_2}^{\alpha+\sum\limits_{i=1}^n(\frac{1}{q_{1i}}-\frac{1}{q_{2i}}),\lambda}(\mathbb{R}^n) \subset  M\dot{K}_{p,\vec{q}_1}^{\alpha+\sum\limits_{i=1}^n(\frac{1}{q_{1i}}-\frac{1}{q_{2i}}),\lambda}(\mathbb{R}^n)$.
\end{proposition}
\begin{proof}
	(1)\quad First, note the obvious inequality, 
$
	\left(\sum\limits_{k=1}^{\infty}|a_k|\right)^r\leq \sum\limits_{k=1}^{\infty}|a_k|^r~ (0<r\leq 1 , a_k\in\mathbb{R}).
$\\
The following result can be concluded through this inequality
	\begin{align*}
	\|f\|_{M\dot{K}_{p_2,\vec{q}}^{\alpha,\lambda}}&=
	\mathop{sup}\limits_{k_0\in \mathbb{Z}}2^{-k_{0}\lambda}
	\left(\sum_{k=-\infty}^{k_0}\left(2^{k\alpha}\|f\chi_k\|_{\vec{q}}\right)^{p_2}\right)^{\frac{1}{p_1} \frac{p_1}{p_2}}\\
	&\leq\mathop{sup}\limits_{k_0\in \mathbb{Z}}2^{-k_{0}\lambda}
	 \left(\sum_{k=-\infty}^{k_0}\left(2^{k\alpha}\|f\chi_k\|_{\vec{q}}\right)^{p_1}\right)^{\frac{1}{p_1}} \leq \|f\|_{M\dot{K}_{p_1,\vec{q}}^{\alpha,\lambda}}.
	\end{align*}
	(2)
	\begin{align*}
||f||_{M\dot{K}_{p,\vec{q}}^{\alpha_2,\lambda}}
&=2^{k(\alpha_2-\alpha_1)}\mathop{sup}\limits_{k_0\in \mathbb{Z}}2^{-k_{0}\lambda}\left(\sum_{k=-\infty}^{k_0}2^{k\alpha_{1} p}||f\widetilde{\chi}_k||_{\vec{q}}^p\right)^{\frac{1}{p}}	
	\leq C ||f||_{M\dot{K}_{p,\vec{q}}^{\alpha_1,\lambda}}.
	\end{align*}
	(3) Using the H$\ddot{o}$lder inequality and condition (2),
	\begin{align*}
	\|f\|_{M\dot{K}_{p,\vec{q}_1}^{\alpha,\lambda}}&= 
	\mathop{sup}\limits_{k_0\in \mathbb{Z}}2^{-k_{0}\lambda}
	\left(\sum_{k=-\infty}^{k_0}2^{k\alpha p}\left(\int_{\mathbb{R}}\dots \left(\int_{\mathbb{R}}\left(\int_{\mathbb{R}}|f\chi_k(x)|^{q_{11}}dx_1\right)^{\frac{q_{12}}{q_{11}}}dx_2\right)^{\frac{q_{13}}{q_{12}}}\dots dx_n\right)^{\frac{p}{q_{1n}}}\right)^{\frac{1}{p}}	\\
    &\leq  \mathop{sup}\limits_{k_0\in \mathbb{Z}}2^{-k_{0}\lambda}
	\left(\sum_{k=-\infty}^{k_0}2^{k\alpha p}\left(\int_{2^{k-1}\leq|x_n|<2^k}\dots \left(\int_{2^{k-1}\leq|x_1|<2^k}|f(x)|^{q_{11}}dx_1\right)^{\frac{q_{12}}{q_{11}}}\dots dx_n\right)^{\frac{p}{q_{1n}}}\right)^{\frac{1}{p}}	\\
	&\leq \mathop{sup}\limits_{k\in \mathbb{Z}}2^{-k_{0}\lambda}
	\left(\sum_{k=-\infty}^{k_0}2^{k(\alpha+\sum\limits_{i=1}^{n}(\frac{1}{q_{1i}}-\frac{1}{q_{2i}}) p}\left(\int_{\mathbb{R}}\dots \left(\int_{\mathbb{R}}\left(\int_{\mathbb{R}}|f\chi_k(x)|^{q_{21}}dx_1\right)^{\frac{q_{22}}{q_{21}}}dx_2\right)^{\frac{q_{23}}{q_{22}}}\dots dx_n\right)^{\frac{p}{q_{2n}}}\right)^{\frac{1}{p}}	\\
	&\leq \|f\|_{M\dot{K}_{p,\vec{q}_2}^{\alpha+\sum\limits_{i=1}^n(\frac{1}{q_{1i}}-\frac{1}{q_{2i}}),\lambda}}.
	\end{align*}
	
\end{proof}
\section{Boundedness of Sublinear Operators and Fractional Type Operators in homogeneous mixed Herz-Morrey spaces}
In this section, the boundedness of the operators is obtained by decomposing the function.
\begin{theorem}\label{Th5.1}
	Let $0\leq \lambda <\infty$, $0<p< \infty$, $1<\vec{q}< \infty$ and $-\sum\limits_{i=1}^n \frac{1}{q_i} +\lambda<\alpha<n\left(1-\frac{1}{n}\sum\limits_{i=1}^n\frac{1}{q_i}\right),$ suppose  sublinear operators $T$ satisfied that\\
	$(1)$ $T$ is bounded on $L^{\vec{q}}\left(\mathbb{R}^n\right);$\\
	$(2)$ for suitable functions $f$ with $suppf\subseteq A_k$ and $|x|\geq2^{k+1}$ with $k\in \mathbb {Z},$
	\begin{equation}
	\mid Tf(x)\mid\leq C\frac{\|f\|_{L^1}}{|x|^n} ;\label{5.1}
	\end{equation}
	$(3)$ for suitable functions $f$ with $suppf\subseteq A_k$ and $|x|\leq 2^{k-2}$ with $k\in \mathbb {Z},$
	\begin{equation}
	\mid Tf(x)\mid\leq C\frac{\|f\|_{L^1}}{2^{nk}} .\label{5.2}
	\end{equation}
	Then $T$ is also bounded on $M\dot{K}_{p,\vec{q}}^{\alpha,\lambda}(\mathbb{R}^n).$
\end{theorem}
\begin{proof}
Assuming	$f\in M\dot{K}_{p,\vec{q}}^{\alpha,\lambda}(\mathbb{R}^n),$   $$f(x)=\sum_{j=-\infty}^{\infty}f(x)\chi_{j}(x)=\sum_{j=-\infty}^{\infty} f _{j}(x).$$
	$$
	\begin{aligned}
	\|Tf\|_{M\dot{K}_{p,\vec{q}}^{\alpha,\lambda}}
	&\leq\mathop{sup}\limits_{k_{0}\in \mathbb{Z}} 2^{-k_{0}\lambda} \left\{\sum_{k=-\infty}^{k_{0}}2^{k \alpha p}
	\left\| \left[\sum ^{\infty}_{j=-\infty}\left(Tf_{j}(\cdot)\right)\right]\chi_{k}(\cdot) \right\| _{L^{\vec{q}}}^p\right\}^{\frac{1}{p}}\\
	\end{aligned}
	$$
	$$
	\begin{aligned}
	\quad\quad\quad\quad\quad\quad
	&\leq\mathop{sup}\limits_{k_{0}\in \mathbb{Z}} 2^{-k_{0}\lambda} \left\{\sum_{k=-\infty}^{k_{0}}2^{k \alpha p}
	\left\| \left[\sum ^{\infty}_{j=-\infty}\left|Tf_{j}(\cdot)\right|\right]\chi_{k}(\cdot) \right\| _{L^{\vec{q}}}^p\right\}^{\frac{1}{p}}\\
	&\leq C\mathop{sup}\limits_{k_{0}\in \mathbb{Z}} 2^{-k_{0}\lambda} \left\{\sum_{k=-\infty}^{k_{0}}2^{k \alpha p}
	\left\| \left[\sum ^{k-2}_{j=-\infty}\left|Tf_{j}(\cdot)\right|\right]\chi_{k}(\cdot) \right\| _{L^{\vec{q}}}^p\right\}^{\frac{1}{p}}\\
				\end{aligned}
	$$
	$$
	\begin{aligned}
\quad \quad \quad~ \quad~\quad~\quad~	&\quad +C\mathop{sup}\limits_{k_{0}\in \mathbb{Z}} 2^{-k_{0}\lambda} \left\{\sum_{k=-\infty}^{k_{0}}2^{k \alpha p}
	\left\| \left[\sum ^{k+1}_{j=k-1}\left|Tf_{j}(\cdot)\right|\right]\chi_{k}(\cdot) \right\| _{L^{\vec{q}}}^p\right\}^{\frac{1}{p}}\\
	&\quad +C\mathop{sup}\limits_{k_{0}\in \mathbb{Z}} 2^{-k_{0}\lambda} \left\{\sum_{k=-\infty}^{k_{0}}2^{k \alpha p}
	\left\| \left[\sum ^{\infty}_{j=k+2}\left|Tf_{j}(\cdot)\right|\right]\chi_{k}(\cdot) \right\| _{L^{\vec{q}}}^p\right\}^{\frac{1}{p}}\\
	&:=C\left(E_1+E_2+E_3\right).
	\end{aligned}
	$$
	For $E_{2}$, $T$ is bounded on  $L^{\overrightarrow{p}}(\mathbb{R}^{n})$, \\
	$$
	\begin{aligned}
	E_2&\leq C\mathop{sup}\limits_{k_{0}\in \mathbb{Z}} 2^{-k_{0}\lambda} \left\{\sum_{k=-\infty}^{k_{0}}2^{k \alpha p}
	\sum ^{k+1}_{j=k-1} \left\|\left|Tf_{j}(\cdot)\right|\right\| _{L^{\vec{q}}}^p\right\}^{\frac{1}{p}}\\
	&\leq C\mathop{sup}\limits_{k_{0}\in \mathbb{Z}} 2^{-k_{0}\lambda} \left\{\sum_{k=-\infty}^{k_{0}}2^{k \alpha p}
	\sum ^{k+1}_{j=k-1} \left\|\left|f_{j}(\cdot)\right| \right\| _{L^{\vec{q}}}^p\right\}^{\frac{1}{p}}\\
	&\leq C\mathop{sup}\limits_{k_{0}\in \mathbb{Z}} 2^{-k_{0}\lambda} \left\{\sum_{k=-\infty}^{k_{0}}2^{k \alpha p}
	\left\|\left|f_{k}(\cdot)\right| \right\| _{L^{\vec{q}}}^p\right\}^{\frac{1}{p}}\\
	&\leq C\|f\|_{M\dot{K}_{p,\vec{q}}^{\alpha,\lambda}}.
	\end{aligned}
	$$
For $E_{1}$,  $j\leq k-2$ and $x\in A_k,$  from condition \eqref{5.1}, 
$$
|T(f_{j}(x)\chi_{k}(x))|\leq C 2^{-kn} \|f_{j}\|_{L^{1}(\mathbb{R}^{n})}
,$$
case 1: $0<p \leq 1$
$$
\begin{aligned}
E_1&\leq C\mathop{sup}\limits_{k_{0}\in \mathbb{Z}} 2^{-k_{0}\lambda} \left\{\sum_{k=-\infty}^{k_{0}}2^{k \alpha p}
\left\|\left[\sum ^{k-2}_{j=-\infty} \| f_{j}\|_{L^{1}} 2^{-kn}\right] \chi_{k}(x) \right\| _{L^{\vec{q}}}^p\right\}^{\frac{1}{p}}\\
&= C\mathop{sup}\limits_{k_{0}\in \mathbb{Z}} 2^{-k_{0}\lambda} \left\{\sum_{k=-\infty}^{k_{0}}2^{k \alpha p}
\left(\sum ^{k-2}_{j=-\infty} \left\|f_{j} \right\| _{L^{1}}|2^{-kn}| \|\chi_{k}\|_{L^{\vec{q}}} \right)^{p}\right\}^{\frac{1}{p}}\\
&\leq C\mathop{sup}\limits_{k_{0}\in \mathbb{Z}} 2^{-k_{0}\lambda} \left\{\sum_{k=-\infty}^{k_{0}}2^{k \alpha p}
\left(\sum ^{k-2}_{j=-\infty} \left\|f_{j} \right\| _{L^{1}}\right)^{p} 2^{-knp} 2^{knp\left( \frac{1}{n} \sum^{n}_{i=1}\frac{1}{q_{i}} \right)} \right\}^{\frac{1}{p}}\\
\end{aligned}
$$
$$
\begin{aligned}
\quad \quad
&\leq C\mathop{sup}\limits_{k_{0}\in \mathbb{Z}} 2^{-k_{0}\lambda} \left\{\sum_{k=-\infty}^{k_{0}}2^{k \alpha p}
\left(\sum ^{k-2}_{j=-\infty} \left\|f_{j} \right\| _{L^{\vec{q}}}\right)^{p}  2^{(k-j)np\left( \frac{1}{n} \sum^{n}_{i=1}\frac{1}{q_{i}} -1 \right)} \right\}^{\frac{1}{p}}\\
&\leq C\mathop{sup}\limits_{k_{0}\in \mathbb{Z}} 2^{-k_{0}\lambda} \left\{\sum_{k=-\infty}^{k_{0}}2^{k \alpha p}
\left\|f_{j} \right\| _{L^{\vec{q}}}^{p} \right\}^{\frac{1}{p}}\\
&\leq C\|f\|_{M\dot{K}_{p,\vec{q}}^{\alpha,\lambda}},
\end{aligned}
$$
case 2: $1<p <\infty $
	$$
\begin{aligned}
E_1&\leq C\mathop{sup}\limits_{k_{0}\in \mathbb{Z}} 2^{-k_{0}\lambda} \left\{\sum_{k=-\infty}^{k_{0}}2^{k \alpha p}
\left[\sum ^{k-2}_{j=-\infty} \left\| f_{j}\right\|_{L^{1}} |2^{-kn}| \right]^{p} \left\|\chi_{k}\right\|  _{L^{\vec{q}}}^p\right\}^{\frac{1}{p}}\\
&\leq C\mathop{sup}\limits_{k_{0}\in \mathbb{Z}} 2^{-k_{0}\lambda} \left\{\sum_{k=-\infty}^{k_{0}}2^{k \alpha p}
\left(\sum ^{k-2}_{j=-\infty} \left\|f_{j} \right\| _{L^{1}} \right)^{p} 2^{-knp} 2^{knp\left( \frac{1}{n} \sum^{n}_{i=1}\frac{1}{q_{i}} \right)} \right\}^{\frac{1}{p}}\\
&\leq  C\mathop{sup}\limits_{k_{0}\in \mathbb{Z}} 2^{-k_{0}\lambda} \left\{\sum_{k=-\infty}^{k_{0}}2^{k \alpha p}\left(\sum_{j=-\infty}^{k-2} \|f\chi_j\|_{L^{\vec{q}}}\right)^p 2^{(k-j)np\left(\frac{1}{n}\sum\limits_{i=1}^n\frac{1}{q_i}-1\right)}\right\}^{\frac{1}{p}}\\
&\leq  C\mathop{sup}\limits_{k_{0}\in \mathbb{Z}} 2^{-k_{0}\lambda} \left\{\sum_{k=-\infty}^{k_{0}}2^{k \alpha p}\left(\sum_{j=-\infty}^{k-2} \|f\chi_j\|_{L^{\vec{q}}} 2^{(k-j)n\left(\frac{1}{n}\sum\limits_{i=1}^n\frac{1}{q_i}-1\right)\frac{1}{2}\times2}\right)^p\right\}^{\frac{1}{p}}\\
&\leq  C\mathop{sup}\limits_{k_{0}\in \mathbb{Z}} 2^{-k_{0}\lambda} \Bigg\{\sum_{k=-\infty}^{k_{0}}2^{k \alpha p}\sum_{j=-\infty}^{k-2} \|f\chi_j\|_{L^{\vec{q}}}^p 2^{(k-j)np\frac{1}{2}\left(\frac{1}{n}\sum\limits_{i=1}^n\frac{1}{q_i}-1\right)}\\
&\quad \quad \quad \times\left(\sum_{j=-\infty}^{k-2}2^{(k-j)np^{\prime}\frac{1}{2}\left(\frac{1}{n}\sum\limits_{j=1}^n\frac{1}{q_i}-1\right)}\right)^\frac{p}{p^{\prime}} \Bigg\}^{\frac{1}{p}}\\
&\leq  C\mathop{sup}\limits_{k_{0}\in \mathbb{Z}} 2^{-k_{0}\lambda} \left\{\sum_{k=-\infty}^{k_{0}}2^{j \alpha p}\|f\chi_j\|_{L^{\vec{q}}}^p\sum_{k=j+2}^{\infty}  2^{(k-j)\left(\alpha -np\frac{1}{2}\left(\frac{1}{n}\sum\limits_{i=1}^n\frac{1}{q_i}-1\right)\right)}\right\}^{\frac{1}{p}}\\
&\leq C\|f\|_{M\dot{K}_{p,\vec{q}}^{\alpha,\lambda}}.
\end{aligned}
$$
That is also why $\alpha<n\left(1-\frac{1}{n}\sum\limits_{i=1}^n\frac{1}{q_i}\right). $
$E_3$ uses the same way, we know $j\geq k+2$ and $x\in A_k,$  through the condition \eqref{5.2}, 
$$
\mid Tf(x)\mid\leq C\frac{\|f\|_{L^1}}{2^{nk}}.
$$
$$
\begin{aligned}
E_3&\leq C\mathop{sup}\limits_{k_{0}\in \mathbb{Z}} 2^{-k_{0}\lambda} \left\{\sum_{k=-\infty}^{k_{0}}2^{k \alpha p}
\left\|\left[\sum ^{\infty}_{j=k+2} \| f_{j}\|_{L^{1}} 2^{-jn}\right] \chi_{k}(x) \right\| _{L^{\vec{q}}}^p\right\}^{\frac{1}{p}}\\
&\leq C\mathop{sup}\limits_{k_{0}\in \mathbb{Z}} 2^{-k_{0}\lambda} \left\{\sum_{k=-\infty}^{k_{0}}2^{k \alpha p}
\left(\sum ^{\infty}_{j=k+2} \left\|f_{j} \right\| _{L^{\vec{q}}}\right)^{p} 2^{(k-j)p\left( \sum^{n}_{i=1}\frac{1}{q_{i}} \right)} \right\}^{\frac{1}{p}},\\
\end{aligned}
$$
case 1: $0<p\leq 1$
$$
\begin{aligned}
E_3&\leq C\mathop{sup}\limits_{k_{0}\in \mathbb{Z}} 2^{-k_{0}\lambda} \left\{\sum_{k=-\infty}^{k_{0}}2^{k \alpha p}
\sum ^{\infty}_{j=k+2} \left\|f_{j} \right\| _{L^{\vec{q}}}^{p} 2^{(k-j)p\left( \sum^{n}_{i=1}\frac{1}{q_{i}} \right)} \right\}^{\frac{1}{p}}\\
&\leq C\mathop{sup}\limits_{k_{0}\in \mathbb{Z}} 2^{-k_{0}\lambda} \left\{\sum_{k=-\infty}^{k_{0}}2^{k \alpha p}
\sum ^{k_{0}}_{j=k+2} \left\|f_{j} \right\| _{L^{\vec{q}}}^{p} 2^{(k-j)p\left( \sum^{n}_{i=1}\frac{1}{q_{i}} \right)} \right\}^{\frac{1}{p}}\\
  &\quad \quad \quad
+C\mathop{sup}\limits_{k_{0}\in \mathbb{Z}} 2^{-k_{0}\lambda} \left\{\sum_{k=-\infty}^{k_{0}}2^{k \alpha p}
\sum ^{\infty}_{j=k_{0}+1} \left\|f_{j} \right\| _{L^{\vec{q}}}^{p} 2^{(k-j)p\left( \sum^{n}_{i=1}\frac{1}{q_{i}} \right)} \right\}^{\frac{1}{p}}\\
&\leq C\mathop{sup}\limits_{k_{0}\in \mathbb{Z}} 2^{-k_{0}\lambda} \left\{\sum_{j=-\infty}^{k_{0}}2^{j \alpha p}
\left\|f_{j} \right\| _{L^{\vec{q}}}^{p} \sum ^{j-2}_{k=-\infty}
2^{(k-j)p\left( \alpha + \sum^{n}_{i=1}\frac{1}{q_{i}} \right)} \right\}^{\frac{1}{p}}\\
&\quad \quad \quad
+C\mathop{sup}\limits_{k_{0}\in \mathbb{Z}} 2^{-k_{0}\lambda} \left\{\sum_{k=-\infty}^{k_{0}}2^{k \alpha p}
\sum ^{\infty}_{j=k_{0}+1} 2^{(k-j)p\left( \sum^{n}_{i=1}\frac{1}{q_{i}} \right)} 2^{-j\alpha p} \sum ^{j}_{l=-\infty} 2^{l\alpha p}\left\|f_{l} \right\| _{L^{\vec{q}}}^{p} \right\}^{\frac{1}{p}}\\
&\leq C\|f\|_{M\dot{K}_{p,\vec{q}}^{\alpha,\lambda}}+C\mathop{sup}\limits_{k_{0}\in \mathbb{Z}} 2^{-k_{0}\lambda} \left\{\sum_{k=-\infty}^{k_{0}}2^{k \alpha p}
\sum ^{\infty}_{j=k_{0}+1} 2^{(k-j)p\left( \sum^{n}_{i=1}\frac{1}{q_{i}} \right)} 2^{-j\alpha p} 2^{j\lambda p}\right\}^{\frac{1}{p}}\\
&\quad\quad\quad\quad\quad\quad\quad\quad\quad\quad\quad\quad
\times\left\{\left[2^{-j\lambda}
\left( \sum ^{j}_{l=-\infty} 2^{l\alpha p} \left\|f_{l} \right\| _{L^{\vec{q}}}^{p}  \right]^{p} \right)^{\frac{1}{p}}  \right\}^{\frac{1}{p}}\\
&\leq C\|f\|_{M\dot{K}_{p,\vec{q}}^{\alpha,\lambda}}+C\|f\|_{M\dot{K}_{p,\vec{q}}^{\alpha,\lambda}}
\mathop{sup}\limits_{k_{0}\in \mathbb{Z}} 2^{-k_{0}\lambda} \left\{\sum_{k=-\infty}^{k_{0}}2^{k \alpha p}
\sum ^{\infty}_{j=k_{0}+1} 2^{(k-j)p\left( \sum^{n}_{i=1}\frac{1}{q_{i}} \right)} 2^{-j\alpha p} 2^{j\lambda p}\right\}^{\frac{1}{p}}\\
&\leq C\|f\|_{M\dot{K}_{p,\vec{q}}^{\alpha,\lambda}},
\end{aligned}
$$
case 2: $1<p<\infty $
$$
\begin{aligned}
E_3&\leq C\mathop{sup}\limits_{k_{0}\in \mathbb{Z}} 2^{-k_{0}\lambda} \left\{\sum_{k=-\infty}^{k_{0}}2^{k \alpha p}
\left(\sum ^{k_{0}}_{j=k+2} \left\|f_{j} \right\| _{L^{\vec{q}}} 2^{(k-j) \sum^{n}_{i=1}\frac{1}{q_{i}} }\right)^{p} \right\}^{\frac{1}{p}}\\
&\quad \quad \quad
+C\mathop{sup}\limits_{k_{0}\in \mathbb{Z}} 2^{-k_{0}\lambda} \left\{\sum_{k=-\infty}^{k_{0}}2^{k \alpha p}
\left(\sum ^{\infty}_{j=k_{0}+1} \left\|f_{j} \right\| _{L^{\vec{q}}} 2^{(k-j)\sum^{n}_{i=1}\frac{1}{q_{i}}}\right)^{p} \right\}^{\frac{1}{p}}\\
&:=C(F_{1}+F_{2}).
\end{aligned}
$$
For $F_1$,
$$
\begin{aligned}
F_1&\leq C\mathop{sup}\limits_{k_{0}\in \mathbb{Z}} 2^{-k_{0}\lambda} \left\{\sum_{k=-\infty}^{k_{0}}
\left(\sum ^{k_{0}}_{j=k+2} 2^{j\alpha p} \left\|f_{j} \right\| _{L^{\vec{q}}}^{p} 2^{(k-j)\left( \alpha+ \sum^{n}_{i=1}\frac{1}{q_{i}}\right)p\frac{1}{2} }\right) \right\}^{\frac{1}{p}}
\left\{\left[\sum ^{k_{0}}_{j=k+2} 2^{(k-j)\left( \alpha+ \sum^{n}_{i=1}\frac{1}{q_{i}}\right)p\frac{1}{2}}  \right]^{\frac{p}{p'}}  \right\}^{\frac{1}{p}}\\
\end{aligned}
$$
$
~~~\leq C\mathop{sup}\limits_{k_{0}\in \mathbb{Z}} 2^{-k_{0}\lambda} \left\{\sum_{j=-\infty}^{k_{0}} 2^{j\alpha p}
\left\|f_{j} \right\| _{L^{\vec{q}}}^{p}
\sum ^{j-2}_{k=-\infty}   2^{(k-j)\left( \alpha+ \sum^{n}_{i=1}\frac{1}{q_{i}}\right)p\frac{1}{2} } \right\}^{\frac{1}{p}}\\
~~~~\leq C\mathop{sup}\limits_{k_{0}\in \mathbb{Z}} 2^{-k_{0}\lambda} \left\{\sum_{j=-\infty}^{k_{0}} 2^{j\alpha p}
\left\|f_{j} \right\| _{L^{\vec{q}}}^{p} \right\}^{\frac{1}{p}}\\
~~~~\leq C\|f\|_{M\dot{K}_{p,\vec{q}}^{\alpha,\lambda}}.
$\\
For $F_2$, according to the H$\ddot{o}$lder inequality and  $\left\| f_{j} \right\|_{L^{\vec{q}}} \leq 2^{-j\alpha p}\sum^{j}_{l=-\infty} 2^{l\alpha p}\left\| f_{l} \right\|_{L^{\vec{q}}}$,
$$
\begin{aligned}
F_2&= C\mathop{sup}\limits_{k_{0}\in \mathbb{Z}} 2^{-k_{0}\lambda} \left\{\sum_{k=-\infty}^{k_{0}}
\left[\sum ^{\infty}_{j=k_{0}+1} 2^{j\alpha } \left\|f_{j} \right\| _{L^{\vec{q}}} 2^{\frac{(k-j)\left( \alpha+ \sum^{n}_{i=1}\frac{1}{q_{i}}+ \lambda\right)}{2}}  2^{(k-j)\left( \alpha+ \sum^{n}_{i=1}\frac{1}{q_{i}}-\lambda\right)} \right]^{p} \right\}^{\frac{1}{p}}\\
&\leq C\mathop{sup}\limits_{k_{0}\in \mathbb{Z}} 2^{-k_{0}\lambda} \left\{\sum_{k=-\infty}^{k_{0}}
\sum ^{\infty}_{j=k_{0}+1} 2^{j\alpha p} \left\|f_{j} \right\| _{L^{\vec{q}}}^{p} 2^{\frac{(k-j)\left( \alpha+ \sum^{n}_{i=1}\frac{1}{q_{i}}+ \lambda\right)p}{2}}   \right\}^{\frac{1}{p}}\\
&\quad\quad\quad\quad\quad\quad
\times\left\{\left[\sum ^{\infty}_{j=k_{0}+1} 2^{\frac{(k-j)\left( \alpha+ \sum^{n}_{i=1}\frac{1}{q_{i}} -\lambda \right)p'}{2}}  \right]^{\frac{p}{p'}}  \right\}^{\frac{1}{p}}\\
&\leq C\mathop{sup}\limits_{k_{0}\in \mathbb{Z}} 2^{-k_{0}\lambda} \left\{\sum_{k=-\infty}^{k_{0}}
\sum ^{\infty}_{j=k_{0}+1} 2^{j\alpha p} \left\|f_{j} \right\| _{L^{\vec{q}}}^{p} 2^{\frac{(k-j)\left( \alpha+ \sum^{n}_{i=1}\frac{1}{q_{i}}+ \lambda\right)p}{2}}   \right\}^{\frac{1}{p}}\\
&\leq C\mathop{sup}\limits_{k_{0}\in \mathbb{Z}} 2^{-k_{0}\lambda} \left\{\sum_{k=-\infty}^{k_{0}}
\sum ^{\infty}_{j=k_{0}+1} 2^{j\alpha p}  2^{\frac{(k-j)\left( \alpha+ \sum^{n}_{i=1}\frac{1}{q_{i}}+ \lambda\right)p}{2}}
\sum^{j}_{l=-\infty} 2^{l\alpha p} \left\|f_{l} \right\| _{L^{\vec{q}}}^{p}
\right\}^{\frac{1}{p}}\\
&= C\mathop{sup}\limits_{k_{0}\in \mathbb{Z}} 2^{-k_{0}\lambda} \left\{\sum_{k=-\infty}^{k_{0}}
\sum ^{\infty}_{j=k_{0}+1} 2^{j\alpha p}  2^{\frac{(k-j)\left( \alpha+ \sum^{n}_{i=1}\frac{1}{q_{i}}+ \lambda\right)p}{2}}
2^{j\lambda p} \left[  2^{-j\lambda} \left(\sum^{j}_{l=-\infty} 2^{l\alpha p} \left\|f_{l} \right\| _{L^{\vec{q}}}^{p} \right)^{\frac{1}{p}} \right]^{p}\right\}^{\frac{1}{p}}\\
&= C\|f\|_{M\dot{K}_{p,\vec{q}}^{\alpha,\lambda}} \mathop{sup}\limits_{k_{0}\in \mathbb{Z}} 2^{-k_{0}\lambda} \left\{\sum_{k=-\infty}^{k_{0}}2^{k\lambda p}
\sum ^{\infty}_{j=k_{0}+1} 2^{\frac{(k-j)\left( \alpha+ \sum^{n}_{i=1}\frac{1}{q_{i}}-\lambda\right)p}{2}}
\right\}^{\frac{1}{p}}\\
&\leq C\|f\|_{M\dot{K}_{p,\vec{q}}^{\alpha,\lambda}} \mathop{sup}\limits_{k_{0}\in \mathbb{Z}} 2^{-k_{0}\lambda} \left\{\sum_{k=-\infty}^{k_{0}}2^{k\lambda p}
\right\}^{\frac{1}{p}}\\
&\leq C\|f\|_{M\dot{K}_{p,\vec{q}}^{\alpha,\lambda}}.
\end{aligned}
$$
That is also why $-\sum\limits_{i=1}^n \frac{1}{q_i} +\lambda<\alpha $, 
$$
E_{3}\leq F_{1}+F_{2}\leq C\|f\|_{M\dot{K}_{p,\vec{q}}^{\alpha,\lambda}}.
$$
All in all
$$
\| Tf\|_{M\dot{K}_{p,\vec{q}}^{\alpha,\lambda}(\mathbb{R}^{n})}\leq C\| f\|_{M\dot{K}_{p,\vec{q}}^{\alpha,\lambda}(\mathbb{R}^{n})}.
$$
This completes the proof of the theorem.
\end{proof}

\begin{remark}
	In the above proof, in order to ensure the convergence of series  $\sum_{k=-\infty}^{k_{0}}2^{k\lambda p} <\infty$, $ \lambda > 0$ is required, but due to $ M\dot{K}_{p,\vec{q}}^{\alpha,0}(\mathbb{R}^n)=\dot{K}_{\vec{q}}^{\alpha, p}(\mathbb{R}^n)$, the result of the theorem in $\dot{K}_{\vec{q}}^{\alpha, p}(\mathbb{R}^n)$   \cite{MW2021}. As proved in, similar situations encountered later will not be explained.
\end{remark}

\begin{corollary}
	Let $0\leq \lambda <\infty$, $0<p< \infty$, $1<\vec{q}< \infty$ and $-\sum\limits_{i=1}^n \frac{1}{q_i} +\lambda<\alpha<n\left(1-\frac{1}{n}\sum\limits_{i=1}^n\frac{1}{q_i}\right)$. If sublinear operators $T$ satisfied the following condition\\
	\begin{equation}
	\left|Tf(x)\right|\leq C\int_{\mathbb{R}}\frac{\left|f(x)\right|}{\left|x-y\right|^n}dy,\quad \quad f\in L^1(\mathbb{R}^n) \mbox{ with compact support}\quad x\notin suppf,\label{7}
	\end{equation}
	and if T is bounded on $L^{\vec{q}}\left(\mathbb{R}^n\right)$, then T is also bounded on $M\dot{K}_{p,\vec{q}}^{\alpha,\lambda}(\mathbb{R}^n)$.
\end{corollary}

\begin{remark}
It is worth noting that \eqref{7} is satisfied by many operators studied in harmonic analysis, such as Calder$\acute{o}$n-Zygmund operators, Hardy-Littlewood maximal operators, R.Fefferman’s singular integral operators and Bochner-Riesz means at the critical index and so on.
\end{remark}

\begin{theorem}
	Let $0<l<n$, $0\leq \lambda < \infty$, $0<p_1\leq p_2\leq \infty$, $1<\vec{q_1}<\frac{1}{l}$, $l=\sum\limits_{i=1}^n\frac{1}{q_{1i}}-\sum\limits_{i=1}^n\frac{1}{q_{2i}}~$and $\lambda-\sum\limits_{i=1}^n\frac{1}{q_{2i}}<\alpha<n-\sum\limits_{i=1}^n\frac{1}{q_{1i}},$ suppose a sublinear operators $I_l$ satisfied that\\
	$(1)$ $I_l$ is bounded from $L^{\vec{q_1}}\left(\mathbb{R}^n\right)$ to $L^{\vec{q_2}}\left(\mathbb{R}^n\right);$\\
	$(2)$ for suitable functions $f$ with $suppf\subset A_j$ and $|x|\geq2^{j+1}$ with $j\in \mathbb {Z},$
	\begin{equation}
	\mid I_lf(x)\mid\leq C\frac{\|f\|_{L^1}}{|x|^{n-l}} ;\label{5.3}
	\end{equation}
	$(3)$ for suitable functions $f$ with $suppf\subset A_j$ and $|x|\leq 2^{j-2}$ with $j\in \mathbb {Z},$
	\begin{equation}
	\mid I_lf(x)\mid\leq C\frac{\|f\|_{L^1}}{2^{j(n-l)}} .\label{5.4}
	\end{equation}
	Then $I_l$ is also bounded from $M\dot{K}_{p_1,\vec{q}_1}^{\alpha,\lambda}(\mathbb{R}^n)$ to $M\dot{K}_{p_2,\vec{q}_2}^{\alpha,\lambda}(\mathbb{R}^n).$
\end{theorem}

\begin{proof}
		Assuming that $0<p<\infty,$ the proof of the case $p=\infty$ is simpler.\\
	Assuming $f\in M\dot{K}_{p_1,\vec{q}_1}^{\alpha,\lambda}(\mathbb{R}^n)$,   $$f(x)=\sum_{j=-\infty}^{\infty}f(x)\chi_{j}(x)=\sum_{j=-\infty}^{\infty} f _{j}(x),$$
	$	\|I_lf\|_{M\dot{K}_{p_2,\vec{q}_2}^{\alpha,\lambda}}\leq \|I_lf\|_{M\dot{K}_{p_1,\vec{q}_2}^{\alpha,\lambda}}
	$ is gained through the Proposition 2.3.\\
Furthermore,                        
	$$
	\begin{aligned}
	\|I_lf\|_{M\dot{K}_{p_2,\vec{q}_2}^{\alpha,\lambda}}^{p_1}
	&\leq\mathop{sup}\limits_{k_{0}\in \mathbb{Z}} 2^{-k_{0}\lambda p_1} \left\{\sum_{k=-\infty}^{k_{0}}2^{k \alpha p_1}
	\left\| |I_lf|\chi_{k} \right\| _{L^{\vec{q}_2}}^{p_1} \right\}\\
	&\leq C\mathop{sup}\limits_{k_{0}\in \mathbb{Z}} 2^{-k_{0}\lambda p_1} \left\{\sum_{k=-\infty}^{k_{0}}2^{k \alpha p_1}
	\left\| \left[\sum ^{k-2}_{j=-\infty}\left|I_l(f_{j})(\cdot)\right|\right]\chi_{k}(\cdot) \right\| _{L^{\vec{q}_{2}}}^{p_{1}}\right\}\\
	&\quad +C\mathop{sup}\limits_{k_{0}\in \mathbb{Z}} 2^{-k_{0}\lambda p_1} \left\{\sum_{k=-\infty}^{k_{0}}2^{k \alpha p_1}
	\left\| \left[\sum ^{k+1}_{j=k-1}\left|I_l(f_{j})(\cdot)\right|\right]\chi_{k}(\cdot) \right\| _{L^{\vec{q}_{2}}}^{p_{1}}\right\}\\
	\end{aligned}
	$$
	$$
	\begin{aligned}
 \quad\quad\quad\quad\quad\quad
	&\quad +C\mathop{sup}\limits_{k_{0}\in \mathbb{Z}} 2^{-k_{0}\lambda p_1} \left\{\sum_{k=-\infty}^{k_{0}}2^{k \alpha p_1}
	\left\| \left[\sum ^{\infty}_{j=k+2}\left|I_l(f_{j})(\cdot)\right|\right]\chi_{k}(\cdot) \right\|_{L^{\vec{q}_{2}}}^{p_{1}}\right\}\\
	&:=C\left(G_1+G_2+G_3\right).
	\end{aligned}
	$$
	For  $G_{2}$,  $I_l$ is bounded from $L^{\vec{q_1}}\left(\mathbb{R}^n\right)$ to $L^{\vec{q_2}}\left(\mathbb{R}^n\right)$,\\
	$$
	\begin{aligned}
	G_2&\leq C\mathop{sup}\limits_{k_{0}\in \mathbb{Z}} 2^{-k_{0}\lambda p_1} \left\{\sum_{k=-\infty}^{k_{0}}2^{k \alpha p_1}\left[\sum ^{k+1}_{j=k-1}
	\left\|\left|I_l(f_{j})(\cdot)\right| \right\|_{L^{\vec{q}_{2}}}\right] ^{p_{1}}\right\}\\
	&\leq C\mathop{sup}\limits_{k_{0}\in \mathbb{Z}} 2^{-k_{0}\lambda p_1} \left\{\sum_{k=-\infty}^{k_{0}}2^{k \alpha p_1}\left[\sum ^{k+1}_{j=k-1}
	\left\|\left|f_{j}\right| \right\|_{L^{\vec{q}_{2}}}\right] ^{p_{1}}\right\}\\
	&\leq C\mathop{sup}\limits_{k_{0}\in \mathbb{Z}} 2^{-k_{0}\lambda p_1} \left\{\sum_{k=-\infty}^{k_{0}}2^{k \alpha p_1}\left\|\left|f_{j}\right| \right\|_{L^{\vec{q}_{1}}} ^{p_{1}}\right\}\\
	&\leq C\|f\|_{M\dot{K}_{p_1,\vec{q}_1}^{\alpha,\lambda}}^{p_1}.
	\end{aligned}
	$$
Afterwards $G_1$, $supp(f\chi_l)\subset A_{k-2},$  $x\in A_k,$ then $2^{k-1}\leq \mid x \mid <2^k,$ through the condition \eqref{5.3}, $$
\begin{aligned}
\mid I_l(f\chi_j)(x) \mid \leq C\frac{\|f\chi_j\|_{L^1}}{|x|^{n-l}}\leq C 2^{-k(n-l)}\|f\chi_j\|_{L^1}.
\end{aligned}$$
Then
$$
\begin{aligned}
G_1&\leq C\mathop{sup}\limits_{k_{0}\in \mathbb{Z}} 2^{-k_{0}\lambda p_1} \left\{\sum_{k=-\infty}^{k_{0}}2^{k \alpha p_1} \left\| \left[\sum ^{k-2}_{j=-\infty}    \left| 2^{-k(n-l)}\right|\left\|f\chi_{j}\right\|_{L^{1}}\right]\chi_k  \right\|_{L^{\vec{q}_{2}}}^{p_{1}}\right\}\\
&\leq C\mathop{sup}\limits_{k_{0}\in \mathbb{Z}} 2^{-k_{0}\lambda p_1} \left\{\sum_{k=-\infty}^{k_{0}}2^{\left(k(\alpha-n+l)+\sum\limits_{i=1}^n\frac{1}{q_{2i}}\right)p_1}\left(\sum\limits_{j=-\infty}^{k-2}\|f\chi_j\|_{L^1}\right)^{p_1}\right\}\\
&\leq C\mathop{sup}\limits_{k_{0}\in \mathbb{Z}} 2^{-k_{0}\lambda p_1} \left\{\sum_{k=-\infty}^{k_{0}}2^{\left(k(\alpha-n+l)+\sum\limits_{i=1}^n\frac{1}{q_{2i}}\right)p_1}\left(\sum\limits_{j=-\infty}^{k-2}\|f\chi_j\|_{L^{\vec{q_1}}}\|\chi_j\|_{L^{\vec{q_1}^{\prime}}}\right)^{p_1}\right\}\\
&\leq C\mathop{sup}\limits_{k_{0}\in \mathbb{Z}} 2^{-k_{0}\lambda p_1} \left\{\sum_{k=-\infty}^{k_{0}}\left[\sum\limits_{j=-\infty}^{k-2}2^{j\alpha}\|f\chi_j\|_{L^{\vec{q_1}}}2^{(j-k)\left(n-\sum\limits_{i=1}^n\frac{1}{q_{1i}}-\alpha\right)}\right]^{p_1}\right\},\\
\end{aligned}$$
case 1: $0<p_1 \leq 1$
	$$
	\begin{aligned}
	G_1&\leq C\mathop{sup}\limits_{k_{0}\in \mathbb{Z}} 2^{-k_{0}\lambda p_1} \left\{\sum_{k=-\infty}^{k_{0}}2^{j\alpha}\|f\chi_j\|^{p_1}_{L^{\vec{q_1}}}\sum\limits_{k=j+2}^{\infty}2^{(j-k)\left(n-\sum\limits_{i=1}^n\frac{1}{q_{1i}}-\alpha\right)p_1}\right\}\\
	&\leq C\mathop{sup}\limits_{k_{0}\in \mathbb{Z}} 2^{-k_{0}\lambda p_1} \left\{\sum_{j=-\infty}^{k_{0}}2^{j \alpha p_1}\left\|\left|f_{j}\right| \right\|_{L^{\vec{q}_{1}}} ^{p_{1}}\right\}\\
	&\leq C\|f\|^{p_1}_{M\dot{K}_{p_1,\vec{q}_1}^{\alpha,\lambda}}.
	\end{aligned}
	$$
That is why $\alpha<n-\sum\limits_{i=1}^n\frac{1}{q_{1i}},$\\
	case 2: $1<p_1 <\infty $
	$$
	\begin{aligned}
	G_1
&\leq  C\mathop{sup}\limits_{k_{0}\in \mathbb{Z}} 2^{-k_{0}\lambda} \Bigg\{\sum_{k=-\infty}^{k_{0}}\Bigg( \sum_{j=-\infty}^{k-2} 2^{j \alpha} \|f\chi_j\|_{L^{\vec{q}}} 2^{(j-k)\frac{1}{2}\left(n-\sum\limits_{i=1}^n\frac{1}{q_{1i}}- \alpha\right)}\\
&\quad \quad \quad \times\sum_{j=-\infty}^{k-2}2^{(j-k)\frac{1}{2}\left(n-\sum\limits_{i=1}^n\frac{1}{q_{1i}}- \alpha\right)} \Bigg)^p \Bigg\}\\
&\leq  C\mathop{sup}\limits_{k_{0}\in \mathbb{Z}} 2^{-k_{0}\lambda} \Bigg\{\sum_{k=-\infty}^{k_{0}}\Bigg( \sum_{j=-\infty}^{k-2} 2^{j p_1 \alpha} \|f\chi_j\|^{p_1}_{L^{\vec{q}}} 2^{(j-k)\frac{p_1}{2}\left(n-\sum\limits_{i=1}^n\frac{1}{q_{1i}}- \alpha\right)}\Bigg)\\
&\quad \quad \quad \times\Bigg(\sum_{j=-\infty}^{k-2}2^{(j-k)\frac{p'_1}{2}\left(n-\sum\limits_{i=1}^n\frac{1}{q_{1i}}- \alpha\right)} \Bigg)^{\frac{p_1}{p'_1}} \Bigg\}\\
&\leq C\mathop{sup}\limits_{k_{0}\in \mathbb{Z}} 2^{-k_{0}\lambda p_1} \left\{\sum_{j=-\infty}^{k_{0}}2^{j \alpha p_1}\left\|\left|f_{j}\right| \right\|_{L^{\vec{q}_{1}}} ^{p_{1}}\right\}\\
&\leq C\|f\|^{p_1}_{M\dot{K}_{p_1,\vec{q}_1}^{\alpha,\lambda}} .
	\end{aligned}
	$$
	That is also why $\alpha<n\left(1-\frac{1}{n}\sum\limits_{i=1}^n\frac{1}{q_i}\right). $
	$E_3$ uses the same way, $j\geq k+2$, $x\in A_k,$  through the condition \eqref{5.4},
	$$
	\begin{aligned}
	\mid I_l(f\chi_j)(x) \mid \leq C\frac{\|f\chi_j\|_{L^1}}{2^{-j(n-l)}}\leq C 2^{-j\sum\limits_{i=1}^n\frac{1}{q_{2i}}}\|f\chi_j\|_{L^{\vec{q_1}}}.
	\end{aligned}
		$$
Then
$$
\begin{aligned}
G_3&\leq C\mathop{sup}\limits_{k_{0}\in \mathbb{Z}} 2^{-k_{0}\lambda p_1} \left\{\sum_{j=-\infty}^{k_{0}}
\left(\sum\limits_{j=k+2}^{\infty}2^{j\alpha}\|f\chi_j\|_{L^{\vec{q_1}}}2^{(k-j)\left(\sum\limits_{i=1}^n\frac{1}{q_{2i}}+\alpha\right)}\right)^{p_1}\right\},
\end{aligned}
$$\\
case 1: $0<p_1\leq 1$
$$
\begin{aligned}
\quad \quad \quad J_3&\leq \mathop{sup}\limits_{k_{0}\in \mathbb{Z}} 2^{-k_{0}\lambda p_1} \left\{\sum_{j=-\infty}^{k_{0}}2^{j\alpha p_1}\|f\chi_j\|_{L^{\vec{q_1}}}^{p_1}\sum\limits_{k=-\infty}^{j-2}2^{(k-j)\left(\sum\limits_{i=1}^n\frac{1}{q_{2i}}+\alpha\right)p_1}\right\}\\
\end{aligned}
$$
$$
\leq\mathop{sup}\limits_{k_{0}\in \mathbb{Z}} 2^{-k_{0}\lambda p_1} \left\{\sum_{j=-\infty}^{k_{0}}2^{j\alpha p_1}\|f\chi_j\|_{L^{\vec{q_1}}}^{p_1}\right\}\\
\leq C\|f\|^{p_1}_{M\dot{K}_{p_1,\vec{q}_1}^{\alpha,\lambda}} ,
$$
case 2:  $1<p_1<\infty$
$$
\begin{aligned}
G_3&\leq C\mathop{sup}\limits_{k_{0}\in \mathbb{Z}} 2^{-k_{0}\lambda p_1} \Bigg\{\sum_{j=-\infty}^{k_{0}}2^{j\alpha p_1}\|f\chi_j\|_{L^{\vec{q_1}}}^{p_1}2^{(k-j)\left(\sum\limits_{i=1}^n\frac{1}{q_{2i}}+\alpha\right)\frac{p_{1}}{2}}\\
&\quad \quad  \times \bigg(\sum\limits_{j=k+2}^{\infty}2^{(k-j)\left(\sum\limits_{i=1}^n\frac{1}{q_{2i}}+\alpha\right)\frac{p_1^{\prime}}{2}}\bigg)^{\frac{p_1}{p_1^{\prime}}}\Bigg\}	\\
				\end{aligned}
$$
$
\begin{aligned}
\quad \quad\quad\quad\quad\quad \quad\quad\quad~   &\leq C\left\{\sum_{j\in \mathbb{Z}}2^{j\alpha p_1}\|f\chi_j\|_{L^{\vec{q_1}}}^{p_1}\right\}\\
&\leq C\|f\|^{p_1}_{M\dot{K}_{p_1,\vec{q}_1}^{\alpha,\lambda}} .
\end{aligned}
$\\
That is why $\lambda-\sum\limits_{i=1}^n\frac{1}{q_{2i}}<\alpha$.\\
In conclusion
$$
\|I_lf\|^{p_1}_{M\dot{K}_{p_2,\vec{q}_2}^{\alpha,\lambda}}\leq C\|f\|^{p_1}_{M\dot{K}_{p_1,\vec{q}_1}^{\alpha,\lambda}}.
$$
Then
$$
\|I_lf\|_{M\dot{K}_{p_2,\vec{q}_2}^{\alpha,\lambda}}\leq C\|f\|_{M\dot{K}_{p_1,\vec{q}_1}^{\alpha,\lambda}}.
$$
This completes the proof of the theorem.
\end{proof}

\begin{corollary}
Let $0<l<n$, $0\leq \lambda < \infty$, $0<p_1\leq p_2\leq \infty$,  $l=\sum\limits_{i=1}^n\frac{1}{q_{1i}}-\sum\limits_{i=1}^n\frac{1}{q_{2i}}$, $1<\vec{q_1}<\frac{1}{l}$ and $\lambda-\sum\limits_{i=1}^n\frac{1}{q_{2i}}<\alpha<n-\sum\limits_{i=1}^n\frac{1}{q_{1i}},$ suppose a sublinear operators $I_l$ satisfied that $I_l$ is bounded from $L^{\vec{q_1}}\left(\mathbb{R}^n\right)$ to $L^{\vec{q_2}}\left(\mathbb{R}^n\right)$ and condition 

	\begin{equation}
	\left|I_lf(x)\right|\leq C\int_{\mathbb{R}}\frac{\left|f(x)\right|}{\left|x-y\right|^{n-l}}dy,\quad \quad f\in L^1(\mathbb{R}^n) \mbox{ with compact support}\quad x\notin suppf,\label{5.5}
    \end{equation}
then $I_l$ is also bounded from $M\dot{K}_{p_1,\vec{q}_1}^{\alpha,\lambda}(\mathbb{R}^n)$ to $M\dot{K}_{p_2,\vec{q}_2}^{\alpha,\lambda}(\mathbb{R}^n)$.
\end{corollary}

\begin{remark}
	It is worth noting that many operators studied in harmonic analysis, such as Risez potential operators and fractional maximum operators, meet the requirements of \eqref{5.5}.
\end{remark}
\section{Boundedness of their commutator operators in homogeneous mixed Herz-Morrey spaces}
In this section, the boundedness for commutators of Hardy-Litttelwood maximal operators, fractional maximal operators, sublinear operators and fractional type operators with BMO functions.
\begin{definition}\cite{SC1989} \cite{SC1993}
	Assuming$~b\in BMO(\mathbb{R}^n)$, Hardy-Litttelwood maximal commutator and fractional maximal commutator respectively defined as
	$$
	M_bf(x)=\mathop{sup}\limits_{r>0}\frac{1}{|B(x,r)|}\int |b(x)-b(y)||f(y)|dy,
	$$
	and
	$$
	M_b^{l}f(x)=\mathop{sup}\limits_{r>0}|B(x,r)|^{-\frac{1}{l^{\prime}}}\int |b(x)-b(y)||f(y)|dy,
	$$
	where $B(x,y)=\{y \in \mathbb{R}^n:|x-y|\leq r\}, 1<l< \infty$ and $\frac{1}{l}+\frac{1}{l^{\prime}}=1 .$
\end{definition}
\begin{theorem}
	Let $0\leq \lambda <\infty$, $0<p< \infty$, $1<\vec{q}< \infty$ and $-\sum\limits_{i=1}^n \frac{1}{q_i} +\lambda<\alpha<n\left(1-\frac{1}{n}\sum\limits_{i=1}^n\frac{1}{q_i}\right),$  suppose sublinear operators $M_bf(x)$ satisfied that  $M_bf(x)$ is bounded on $L^{\vec{q}}\left(\mathbb{R}^n\right)and ~b\in BMO(\mathbb{R}^n).$
	Then $M_bf(x)$ is also bounded on $M\dot{K}_{p,\vec{q}}^{\alpha,\lambda}\left(\mathbb{R}^n\right).$
\end{theorem}
\begin{proof}
    Assuming $f \in M \dot{K}_{p, \vec{q}}^{\alpha, \lambda}\left(\mathbb{R}^{n}\right)$,
	$$
	f(x)=\sum_{j=-\infty}^{\infty} f(x) \chi_{j}(x) \equiv \sum_{j=-\infty}^{\infty} f_{j}(x).
	$$
	$$
	\begin{aligned}
	\left\|M_{b}(f)\right\|_{M \dot{K}_{p,  \vec{q}}^{\alpha, \lambda}} \leq & C \sup _{k_{0} \in \mathbb{Z}} 2^{-k_{0} \lambda}\left\{\sum_{k=-\infty}^{k_{0}} 2^{k \alpha p}\left(\sum_{j=-\infty}^{k-3}\left\|\left(M_{b}\left(f_{j}\right)(\cdot)\right) \chi_{k}(\cdot)\right\|_{L^{ \vec{q}}}\right)^{p}\right\}^{1 / p} \\
	&+C \sup _{k_{0} \in \mathbb{Z}} 2^{-k_{0} \lambda}\left\{\sum_{k=-\infty}^{k_{0}} 2^{k \alpha p}\left(\sum_{j=k-2}^{k+2}\left\|\left(M_{b}\left(f_{j}\right)(\cdot)\right) \chi_{k}(\cdot)\right\|_{L^{\vec{q}}}\right)^{p}\right\}^{1 / p} \\
	&+C \sup _{k_{0} \in \mathbb{Z}} 2^{-k_{0} \lambda}\left\{\sum_{k=-\infty}^{k_{0}} 2^{k \alpha p}\left(\sum_{j=k+3}^{\infty}\left\|\left(M_{b}\left(f_{j}\right)(\cdot)\right) \chi_{k}(\cdot)\right\|_{L^{\vec{q}}}\right)^{p}\right\}^{1 / p} \\
	:= & C\left(D_{1}+D_{2}+D_{3}\right) .
	\end{aligned}
	$$
	For~$D_{2}$,  $M_{b}$ is bounded on $L^{\vec{q}}\left(\mathbb{R}^{n}\right)$\cite{TN2017},
	$$
	\begin{aligned}
	D_{2} & \leq C \sup _{k_{0} \in \mathbb{Z}} 2^{-k_{0} \lambda}\left\{\sum_{k=-\infty}^{k_{0}} 2^{k \alpha p}\left(\sum_{j=k-2}^{k+2}\left\|f_{j}\right\|_{L^{\vec{q}}}\right)^{p}\right\}^{1 / p} \\
	& \leq C \sup _{k_{0} \in \mathbb{Z}} 2^{-k_{0} \lambda}\left\{\sum_{k=-\infty}^{k_{0}} 2^{k \alpha p}\left\|f_{k}\right\|_{L^{\vec{q}}}^{p}\right\}^{1 / p} \leq C\|f\|_{M \dot{K}_{p, \vec{q}}^{\alpha, \lambda}} .
	\end{aligned}
	$$
For $D_1$, the average value of $b(x)$ on $B(0,2^j)$ is recorded as $b_j$. According to the nature of $BMO(\mathbb{R}^n)$  \cite{EM1993} and note that $j \leq k-3$, it can be deduced
	$$
	\begin{aligned}
	\left\|\left(M_{b} f_{j}\right) \chi_{k}\right\|_{L^{\vec{q}}} 
	&\leq C\|b\|_{\mathrm{BMO}} 2^{(j-k) n\left(1-\frac{1}{n}\sum\limits_{i=1}^n\frac{1}{q_{i}}\right)}(k-j)\left\|f_{j}\right\|_{L^{\vec{q}}}\\ 
	&\leq C 2^{(j-k) n\left(1-\frac{1}{n}\sum\limits_{i=1}^n\frac{1}{q_{i}}\right)}(k-j)\left\|f_{j}\right\|_{L^{\vec{q}}} .
	\end{aligned}
	$$
	Then
	$$
	\begin{aligned}
	D_{1} & \leq C \sup _{k_{0} \in \mathbb{Z}} 2^{-k_{0} \lambda}\left\{\sum_{k=-\infty}^{k_{0}} 2^{k \alpha p}\left(\sum_{j=-\infty}^{k-3} 2^{(j-k)n\left(1-\frac{1}{n}\sum\limits_{i=1}^n\frac{1}{q_{i}}\right)}(k-j)\left\|f_{j}\right\|_{L^{\vec{q}}}\right)^{p}\right\}^{1 / p} \\
	&\leq C \sup _{k_{0} \in \mathbb{Z}} 2^{-k_{0} \lambda}\left\{\sum_{k=-\infty}^{k_{0}}\left(\sum_{j=-\infty}^{k-3} 2^{j \alpha}\left\|f_{j}\right\|_{L^{\vec{q}}} 2^{(k-j)\left[\alpha-n\left(1-\frac{1}{n}\sum\limits_{i=1}^n\frac{1}{q_{i}}\right)\right]}(k-j)\right)^{p}\right\}^{1 / p},\\
	\end{aligned}
	$$
case: $0<p \leq 1$
	$$
\begin{aligned}
D_1&\leq C \sup _{k_{0} \in \mathbb{Z}} 2^{-k_{0} \lambda}\left\{\sum_{k=-\infty}^{k_{0}}\sum_{j=-\infty}^{k-3} 2^{j \alpha p}\left\|f_{j}\right\|_{L^{\vec{q}}}^p 2^{(k-j)\left[\alpha-n\left(1-\frac{1}{n}\sum\limits_{i=1}^n\frac{1}{q_{i}}\right)\right]p}(k-j)^{p}\right\}^{1 / p} \\
\end{aligned}
$$
$$
\begin{aligned}
\quad \quad &\leq C \sup _{k_{0} \in \mathbb{Z}} 2^{-k_{0} \lambda}\left\{\sum_{j=-\infty}^{k_{0}}2^{j \alpha p} \left\|f_{j}\right\|_{L^{\vec{q}}}^p \sum_{k=j+3}^{k_0}  2^{(k-j)\left[\alpha-n\left(1-\frac{1}{n}\sum\limits_{i=1}^n\frac{1}{q_{i}}\right)\right]p}(k-j)^{p}\right\}^{1 / p}\\
& \leq C \sup _{k_{0} \in \mathbb{Z}} 2^{-k_{0} \lambda}\left\{\sum_{k=-\infty}^{k_{0}} 2^{k \alpha p}\left\|f_{k}\right\|_{L^{\vec{q}}}^{p}\right\}^{1 / p} \\
&\leq C\|f\|_{M \dot{K}_{p, \vec{q}}^{\alpha, \lambda}},
\end{aligned}
$$
case: $1<p< \infty$
	$$
\begin{aligned}
D_1&\leq C \sup _{k_{0} \in \mathbb{Z}} 2^{-k_{0} \lambda}
\Bigg\{\sum_{k=-\infty}^{k_{0}}
\left[\sum_{j=-\infty}^{k-3} 2^{j \alpha p}\left\|f_{j}\right\|_{L^{\vec{q}}}^p 2^{(k-j)\left[\alpha-n\left(1-\frac{1}{n}\sum\limits_{i=1}^n\frac{1}{q_{i}}\right)\right]\frac{p}{2}}\right] \\
&\quad\quad\quad\quad
\times \left[    \sum_{j=-\infty}^{k-3} (k-j)^{p^{\prime}} 2^{(k-j)\left[\alpha-n\left(1-\frac{1}{n}\sum\limits_{i=1}^n\frac{1}{q_{i}}\right)\right]\frac{p^\prime}{2}}  \right]^{\frac{p}{p^\prime}} \Bigg\}^{\frac{1}{p}}\\
&\leq C \sup _{k_{0} \in \mathbb{Z}} 2^{-k_{0} \lambda}\left\{\sum_{j=-\infty}^{k_{0}}2^{j \alpha p} \left\|f_{j}\right\|_{L^{\vec{q}}}^p \sum_{k=j+3}^{k_0}  2^{(k-j)\left[\alpha-n\left(1-\frac{1}{n}\sum\limits_{i=1}^n\frac{1}{q_{i}}\right)\right]\frac{p}{2}}\right\}^{1 / p}\\
& \leq C \sup _{k_{0} \in \mathbb{Z}} 2^{-k_{0} \lambda}\left\{\sum_{k=-\infty}^{k_{0}} 2^{k \alpha p}\left\|f_{k}\right\|_{L^{\vec{q}}}^{p}\right\}^{1 / p} \\
&\leq C\|f\|_{M \dot{K}_{p, \vec{q}}^{\alpha, \lambda}} .
\end{aligned}
$$
That is also why $\alpha<n\left(1-\frac{1}{n}\sum\limits_{i=1}^n\frac{1}{q_i}\right). $
$D_3$ uses the same way, $j> k+2$, $x\in A_k,$ 
	$$
	\begin{aligned}
	M_{b}\left(f_{j}\right)(x) & \leq C 2^{-j n} \int_{A_{j}}|b(x)-b(y)| \left|f_{j}(y)\right| \mathrm{d} y \\
	& \leq C 2^{-j n}\left|b(x)-b_{k}\right| \left\|f_{j}\right\|_{L^{1}}+C 2^{-j n} \int_{A_{j}}\left|b_{k}-b(y)\right| |f(y)| \mathrm{d} y .
	\end{aligned}
	$$
Then
	$$
	\begin{aligned}
	\left\|\left(M_{b} f_{j}\right) \chi_{k}\right\|_{L^{\vec{q}}} & \leq C\|b\|_{\mathrm{BMO}} 2^{(k-j) \sum\limits_{i=1}^n\frac{1}{q_{i}}}\left\|f_{j}\right\|_{L^{\vec{q}}}+C\|b\|_{\mathrm{BMO}} 2^{(k-j) \sum\limits_{i=1}^n\frac{1}{q_{i}}}(j-k)\left\|f_{j}\right\|_{L^{\vec{q}}}\\
	& \leq C\|b\|_{\mathrm{BMO}} 2^{(k-j) \sum\limits_{i=1}^n\frac{1}{q_{i}}}(j-k)\left\|f_{j}\right\|_{L^{\vec{q}}}\\ 
	&\leq C 2^{(k-j) \sum\limits_{i=1}^n\frac{1}{q_{i}}}(j-k)\left\|f_{j}\right\|_{L^{\vec{q}}} .
	\end{aligned}
	$$

	$$
	\begin{aligned}
	D_{3} & \leq C \sup _{k_{0} \in \mathbb{Z}} 2^{-k_{0} \lambda}\left\{\sum_{k=-\infty}^{k_{0}} 2^{k \alpha p}\left(\sum_{j=k+3}^{\infty} 2^{(k-j) \sum\limits_{i=1}^n\frac{1}{q_{i}}}(j-k)\left\|f_{j}\right\|_{L^{\vec{q}}}\right)^{p}\right\}^{1 / p} \\
	&=C \sup _{k_{0} \in \mathbb{Z}} 2^{-k_{0} \lambda}\left\{\sum_{k=-\infty}^{k_{0}}\left(\sum_{j=k+3}^{\infty} 2^{j \alpha}\left\|f_{j}\right\|_{L^{\vec{q}}}(j-k) 2^{-(j-k)\left(\alpha+\sum\limits_{i=1}^n\frac{1}{q_{i}}\right)}\right)^{p}\right\}^{1 / p},
	\end{aligned}
	$$
case: $0<p \leq 1$, use this inequality  $\left(\sum_{i=- \infty}^{\infty} |a_j|\right)^p \leq \sum_{i=- \infty}^{\infty} |a_j|^p, $ 
	$$
	\begin{aligned}
	D_{3} \leq C & \sup _{k_{0} \in \mathbb{Z}} 2^{-k_{0} \lambda}\left\{\sum_{k=-\infty}^{k_{0}} \sum_{j=k+3}^{\infty} 2^{j \alpha p}\left\|f_{j}\right\|_{L^{\vec{q}}}^{p}(j-k)^{p} 2^{-(j-k)\left(\alpha+\sum\limits_{i=1}^n\frac{1}{q_{i}}\right)p}\right\}^{1 / p} \\
	\leq & C \sup _{k_{0} \in \mathbb{Z}} 2^{-k_{0} \lambda}\left\{\sum_{k=-\infty}^{k_0} \sum_{j=k+3}^{k_{0}} 2^{j \alpha p}\left\|f_{j}\right\|_{L^{\vec{q}}}^{p}(j-k)^{p} 2^{-(j-k)\left(\alpha+\sum\limits_{i=1}^n\frac{1}{q_{i}}\right) p}\right\}^{1 / p} \\
	&+C \sup _{k_{0} \in \mathbb{Z}} 2^{-k_{0} \lambda}\left\{\sum_{k=-\infty}^{k_{0}} \sum_{j=k_{0}+1}^{\infty} 2^{j \alpha p}\left\|f_{j}\right\|_{L^{\vec{q}}}^{p}(j-k)^{p} 2^{-(j-k)\left(\alpha+\sum\limits_{i=1}^n\frac{1}{q_{i}}\right) p}\right\}^{1 / p} \\
:= & C\left(K_{1}+K_{2}\right) .
	\end{aligned}
	$$
	For $K_1$
	$$
	\begin{aligned}
	K_{1} &=C \sup _{k_{0} \in \mathbb{Z}} 2^{-k_{0} \lambda}\left\{\sum_{j=-\infty}^{k_{0}} 2^{j \alpha p}\left\|f_{j}\right\|_{L^{\vec{q}}}^{p} \sum_{k=-\infty}^{j-3}(j-k)^{p} 2^{-(j-k)\left(\alpha+\sum\limits_{i=1}^n\frac{1}{q_{i}}\right)  p}\right\}^{1 / p} \\
& \leq C \sup _{k_{0} \in \mathbb{Z}} 2^{-k_{0} \lambda}\left\{\sum_{k=-\infty}^{k_{0}} 2^{k \alpha p}\left\|f_{k}\right\|_{L^{\vec{q}}}^{p}\right\}^{1 / p} \\
&\leq C\|f\|_{M \dot{K}_{p, \vec{q}}^{\alpha, \lambda}} .
	\end{aligned}
	$$
	For $K_2$, note the obvious inequality $\left\|f_{j}\right\|_{L^{\vec{q}}}^{p} \leq 2^{-j \alpha p} \sum_{l=-\infty}^{j} 2^{l \alpha p}\left\|f_{l}\right\|_{L^{\vec{q}}}^{p}$,
$$
\begin{aligned}
K_{2} &=C \sup _{k_{0} \in \mathbb{Z}} 2^{-k_{0} \lambda}\left\{\sum_{k=-\infty}^{k_{0}} 2^{\left(kp\alpha+kp\sum\limits_{i=1}^n\frac{1}{q_{i}}\right)} \sum_{j=k_{0}+1}^{\infty}(j-k)^{p} 2^{-jp\sum\limits_{i=1}^n\frac{1}{q_{i}}}\left\|f_{j}\right\|_{L^{\vec{q}}}^{p}\right\}^{1 / p} \\
& \leq C \sup _{k_{0} \in \mathbb{Z}} 2^{-k_{0} \lambda}\left\{\sum_{k=-\infty}^{k_{0}} 2^{\left(kp\alpha+kp\sum\limits_{i=1}^n\frac{1}{q_{i}}\right)} \sum_{j=k_{0}+1}^{\infty}(j-k)^{p} 2^{-j p\sum\limits_{i=1}^n\frac{1}{q_{i}}} 2^{-j \alpha p} \sum_{l=-\infty}^{j} 2^{l \alpha p}\left\|f_{i}\right\|_{L^{\vec{q}}}^{p}\right\}^{1 / p} \\
\end{aligned}
$$
$$
\begin{aligned}\quad \quad ~~
&=C \sup _{k_{0} \in \mathbb{Z}} 2^{-k_{0} \lambda}\left\{\sum_{k=-\infty}^{k_{0}} 2^{\left(kp\alpha+kp\sum\limits_{i=1}^n\frac{1}{q_{i}}\right)} \sum_{j=k_{0}+1}^{\infty}(j-k)^{p} 2^{-j \sum\limits_{i=1}^n\frac{1}{q_{i}}} 2^{-j \alpha p} 2^{j \lambda p}\left[2^{-j \lambda}\left(\sum_{l=-\infty}^{j} 2^{l \alpha p}\left\|f_{l}\right\|_{L^{\vec{q}}}^{p}\right)^{1 / p}\right]^{p}\right\}^{1 / p} \\
& \leq C \sup _{k_{0} \in \mathbb{Z}} 2^{-k_{0} \lambda}\left\{\sum_{k=-\infty}^{k_{0}} 2^{k \alpha p}\left\|f_{k}\right\|_{L^{\vec{q}}}^{p}\right\}^{1 / p} \\
&\leq C\|f\|_{M \dot{K}_{p, \vec{q}}^{\alpha,\lambda}},
\end{aligned}
$$
case: $1<p< \infty$
$$
\begin{aligned}
	D_{3} 
	\leq & C \sup _{k_{0} \in \mathbb{Z}} 2^{-k_{0} \lambda}\left\{\sum_{k=-\infty}^{k_0} \left[\sum_{j=k+3}^{k_{0}} 2^{j \alpha }\left\|f_{j}\right\|_{L^{\vec{q}}}(j-k) 2^{-(j-k)\left(\alpha+\sum\limits_{i=1}^n\frac{1}{q_{i}}\right) }\right]^p\right\}^{1 / p} \\
	\end{aligned}
	$$
	$$
	\begin{aligned}
	\quad \quad\quad \quad\quad \quad&+C \sup _{k_{0} \in \mathbb{Z}} 2^{-k_{0} \lambda}\left\{\sum_{k=-\infty}^{k_{0}} \left[\sum_{j=k_{0}+1}^{\infty} 2^{j \alpha }\left\|f_{j}\right\|_{L^{\vec{q}}}(j-k) 2^{-(j-k)\left(\alpha+\sum\limits_{i=1}^n\frac{1}{q_{i}}\right) }\right]^p\right\}^{1 / p} \\
	:= & C\left(E_{1}+E_{2}\right) .
\end{aligned}
$$
For$~E_1$
$$
\begin{aligned}
	E_{1} & \leq \sup _{k_{0} \in \mathbb{Z}} 2^{-k_{0} \lambda}\left\{\sum_{k=-\infty}^{k_{0}}\left[\sum_{j=k+3}^{k_{0}} 2^{j \alpha p}\left\|f_{j}\right\|_{L^{\vec{q}}}^{p} 2^{-(j-k)\left(\alpha+\sum\limits_{i=1}^n\frac{1}{q_{i}} \right) p / 2}\right]\left[\sum_{j=k+3}^{k_{0}}(j-k)^{p^{\prime}} 2^{-(j-k)\left(\alpha+\sum\limits_{i=1}^n\frac{1}{q_{i}} \right) p^{\prime} / 2}\right]^{p / p^{\prime}}\right\}^{1 / p} \\
	& \leq C \sup _{k_{0} \in \mathbb{Z}} 2^{-k_{0} \lambda}\left\{\sum_{k=-\infty}^{k_{0}} \sum_{j=k+3}^{k_{0}} 2^{j \alpha p}\left\|f_{j}\right\|_{L^{\vec{q}}}^{p} 2^{-(j-k)\left(\alpha+\sum\limits_{i=1}^n\frac{1}{q_{i}} \right) p / 2}\right\}^{1 / p} \\
	& \leq C \sup _{k_{0} \in \mathbb{Z}} 2^{-k_{0} \lambda}\left\{\sum_{j=-\infty}^{k_{0}} 2^{j \alpha p}\left\|f_{j}\right\|_{L^{\vec{q}}}^{p} \sum_{k=-\infty}^{j-3} 2^{-(j-k)\left(\alpha+\sum\limits_{i=1}^n\frac{1}{q_{i}} \right) p / 2}\right\}^{1 / p} \\
& \leq C \sup _{k_{0} \in \mathbb{Z}} 2^{-k_{0} \lambda}\left\{\sum_{k=-\infty}^{k_{0}} 2^{k \alpha p}\left\|f_{k}\right\|_{L^{\vec{q}}}^{p}\right\}^{1 / p} \\
&\leq C\|f\|_{M \dot{K}_{p, \vec{q}}^{\alpha,\lambda}} .
\end{aligned}
$$
For $E_{2}$, notice the obvious inequality $\left\|f_{j}\right\|_{L^{q}\left(\mathbb{R}^{n}\right)}^{p} \leq 2^{-j \alpha p} \sum_{l=-\infty}^{j} 2^{l \alpha p}\left\|f_{l}\right\|_{L^{q}\left(\mathbb{R}^{n}\right)}^{p}$, 
$$
\begin{aligned}
E_{2} &=\sup _{k_{0} \in \mathbb{Z}} 2^{-k_{0} \lambda}\left\{\sum_{k=-\infty}^{k_{0}}\left(\sum_{j=k_{0}+1}^{\infty} 2^{j \alpha}\left\|f_{j}\right\|_{L^{\vec{q}}} 2^{(k-j)\left(\alpha+\sum\limits_{i=1}^n\frac{1}{q_{i}} +\lambda\right) / 2}(j-k) 2^{(k-j)\left(\alpha+\sum\limits_{i=1}^n\frac{1}{q_{i}} -\lambda\right) / 2}\right)^{p}\right\}^{1 / p} \\
& \leq C \sup _{k_{0} \in \mathbb{Z}} 2^{-k_{0} \lambda}\left\{\sum_{k=-\infty}^{k_{0}}\left[\sum_{j=k_{0}+1}^{\infty} 2^{j \alpha p}\left\|f_{j}\right\|_{L^{\vec{q}}}^{p} 2^{  \frac{-(j-k)\left(\alpha+\sum\limits_{i=1}^n\frac{1}{q_{i}} +\lambda\right) p}{2}}\right]\left[\sum_{j=k_{0}+1}^{\infty}(j-k)^{p^{\prime}} 2^{\frac{-(j-k)\left(\alpha+\sum\limits_{i=1}^n\frac{1}{q_{i}} -\lambda\right) p^{\prime} }{2}}\right]^{\frac{p}{p^{\prime}} }\right\}^{\frac{1 }{p}} \\
& \leq C \sup _{k_{0} \in \mathbb{Z}} 2^{-k_{0} \lambda}\left\{\sum_{k=-\infty}^{k_{0}} \sum_{j=k_{0}+1}^{\infty} 2^{j \alpha p}\left\|f_{j}\right\|_{L^{\vec{q}}}^{p} 2^{-(j-k)\left(\alpha+\sum\limits_{i=1}^n\frac{1}{q_{i}} +\lambda\right) p / 2}\right\}^{1 / p} \\
& \leq C \sup _{k_{0} \in \mathbb{Z}} 2^{-k_{0} \lambda}\left\{\sum_{k=-\infty}^{k_{0}} \sum_{j=k_{0}+1}^{\infty} 2^{-(j-k)\left(\alpha+\sum\limits_{i=1}^n\frac{1}{q_{i}} +\lambda\right) p / 2} \sum_{l=-\infty}^{j} 2^{l \alpha p}\left\|f_{l}\right\|_{L^{\vec{q}}}^{p}\right\}^{1 / p} \\
&=C \sup _{k_{0} \in \mathbb{Z}} 2^{-k_{0} \lambda}\left\{\sum_{k=-\infty}^{k_{0}} \sum_{j=k_{0}+1}^{\infty} 2^{-(j-k)\left(\alpha+\sum\limits_{i=1}^n\frac{1}{q_{i}} +\lambda\right) p / 2} 2^{j \lambda p}\left[2^{-j \lambda}\left(\sum_{l=-\infty}^{j} 2^{l \alpha p}\left\|f_{l}\right\|_{L^{\vec{q}}}^{p}\right)^{1 / p}\right]^{p}\right\}^{1 / p} \\
&=C\leq C\|f\|_{M \dot{K}_{p, \vec{q}}^{\alpha,\lambda}} \sup _{k_{0} \in \mathbb{Z}} 2^{-k_{0} \lambda}\left\{\sum_{k=-\infty}^{k_{0}} 2^{k \lambda p} \sum_{j=k_{0}+1}^{\infty} 2^{-(j-k)\left(\alpha+\sum\limits_{i=1}^n\frac{1}{q_{i}} -\lambda\right) p / 2}\right\}^{1 / p} \\
&\leq C\|f\|_{M\dot{K}_{p,\vec{q}}^{\alpha,\lambda}} \mathop{sup}\limits_{k_{0}\in \mathbb{Z}} 2^{-k_{0}\lambda} \left\{\sum_{k=-\infty}^{k_{0}}2^{k\lambda p}
\right\}^{\frac{1}{p}}\\
\end{aligned}
$$
$
~~~~\leq C\|f\|_{M\dot{K}_{p, \vec{q}}^{\alpha,\lambda}}.
$\\
That is why $\lambda-\sum\limits_{i=1}^n\frac{1}{q_{1i}}<\alpha .$\\
This completes the proof of the theorem.
\end{proof}
\begin{theorem}
	Let $1<l<\infty$, $0\leq \lambda < \infty$, $0<p_1\leq p_2\leq \infty$, $1<\vec{q_1}<l$, $\frac{1}{l}=\frac{1}{n}\sum\limits_{i=1}^n\frac{1}{q_{1i}}-\frac{1}{n}\sum\limits_{i=1}^n\frac{1}{q_{2i}}~$ and  $\lambda-\sum\limits_{i=1}^n\frac{1}{q_{2i}}<\alpha<n-\sum\limits_{i=1}^n\frac{1}{q_{1i}},$ suppose sublinear operators $M_b^{l}f(x)$ satisfied that $M_b^{l}f(x)$ is bounded from $L^{\vec{q_1}}\left(\mathbb{R}^n\right)$ to $L^{\vec{q_2}}\left(\mathbb{R}^n\right)$and $b\in BMO(\mathbb{R}^n)$.
	Then $M_b^{l}f(x)$ is also bounded from $M\dot{K}_{p_1,\vec{q}_1}^{\alpha,\lambda}(\mathbb{R}^n)$ to $M\dot{K}_{p_2,\vec{q}_2}^{\alpha,\lambda}(\mathbb{R}^n).$
\end{theorem}

\begin{proof}
	Assuming $f \in M \dot{K}_{p, \vec{q}}^{\alpha, \lambda}\left(\mathbb{R}^{n}\right)$,
	$$
	f(x)=\sum_{j=-\infty}^{\infty} f(x) \chi_{j}(x) \equiv \sum_{j=-\infty}^{\infty} f_{j}(x) .
	$$
	the result is concluded through the Proposition 2.3,
	$$	\|M_b^{l}f\|_{M\dot{K}_{p_2,\vec{q}_2}^{\alpha,\lambda}}\leq \|M_b^{l}f\|_{M\dot{K}_{p_1,\vec{q}_2}^{\alpha,\lambda}}.
	$$
	Furthermore
	$$
	\begin{aligned} 
	\|M_b^{l}f\|_{M\dot{K}_{p_2,\vec{q}_2}^{\alpha,\lambda}}^{p_1}
	&\leq\mathop{sup}\limits_{k_{0}\in \mathbb{Z}} 2^{-k_{0}\lambda p_1} \left\{\sum_{k=-\infty}^{k_{0}}2^{k \alpha p_1}
	\left\| M_b^{l}f\chi_{k} \right\| _{L^{\vec{q}_2}}^{p_1} \right\}\\
	&\leq C\mathop{sup}\limits_{k_{0}\in \mathbb{Z}} 2^{-k_{0}\lambda p_1} \left\{\sum_{k=-\infty}^{k_{0}}2^{k \alpha p_1}\sum ^{k-3}_{j=-\infty}
	\left\| M_b^{l}(f_{j})(\cdot)\chi_{k}(\cdot) \right\| _{L^{\vec{q}_{2}}}^{p_{1}}\right\}\\
	&\quad +C\mathop{sup}\limits_{k_{0}\in \mathbb{Z}} 2^{-k_{0}\lambda p_1} \left\{\sum_{k=-\infty}^{k_{0}}2^{k \alpha p_1}\sum ^{k+2}_{j=k-2}
	\left\| M_b^{l}(f_{j})(\cdot)\chi_{k}(\cdot) \right\| _{L^{\vec{q}_{2}}}^{p_{1}}\right\}\\
	\end{aligned}
	$$
	$$
	\begin{aligned}
	\quad\quad\quad\quad\quad\quad
	&\quad +C\mathop{sup}\limits_{k_{0}\in \mathbb{Z}} 2^{-k_{0}\lambda p_1} \left\{\sum_{k=-\infty}^{k_{0}}2^{k \alpha p_1}\sum ^{\infty}_{j=k+3}
	\left\| M_b^{l}(f_{j})(\cdot)\chi_{k}(\cdot) \right\|_{L^{\vec{q}_{2}}}^{p_{1}}\right\}\\
	&:=C\left(F_1+F_2+F_3\right).
	\end{aligned}
	$$
	For  $F_{2}$, $M_b^{l}$ is bounded from $L^{\vec{q_1}}\left(\mathbb{R}^n\right)$ to $L^{\vec{q_2}}\left(\mathbb{R}^n\right)$\cite{TN2017},\\
	$$
	\begin{aligned}
	F_2&\leq C\mathop{sup}\limits_{k_{0}\in \mathbb{Z}} 2^{-k_{0}\lambda p_1} \left\{\sum_{k=-\infty}^{k_{0}}2^{k \alpha p_1}\left\|\left|f_{j}\right| \right\|_{L^{\vec{q}_{1}}} ^{p_{1}}\right\}
	&\leq C\|f\|_{M\dot{K}_{p_1,\vec{q}_1}^{\alpha,\lambda}}^{p_1}.
	\end{aligned}
	$$
For $D_1$, note that $j \leq k-3$, it can be deduced
$$
\begin{aligned}
\left\|M_b^{l} f_{j} \chi_{k}\right\|_{L^{\vec{q}_2}}
&\leq C\|b\|_{\mathrm{BMO}\left(\mathbb{R}^{n}\right)} 2^{(j-k) n\left(1-\frac{1}{n}\sum\limits_{i=1}^n\frac{1}{q_{i}}\right)}(k-j)\left\|f_{j}\right\|_{L^{\vec{q}_1}}\\ 
&\leq C 2^{(j-k) n\left(1-\frac{1}{n}\sum\limits_{i=1}^n\frac{1}{q_{i}}\right)}(k-j)\left\|f_{j}\right\|_{L^{\vec{q}_1}}.
\end{aligned}
$$
So repeat the same process as $D_1$,
$$
F_1 \leq C\mathop{sup}\limits_{k_{0}\in \mathbb{Z}} 2^{-k_{0}\lambda p_1} \left\{\sum_{k=-\infty}^{k_{0}}2^{k \alpha p_1} \sum_{j=-\infty}^{k-3} (k-j)^{p_1}2^{(j-k)n\left(1- \frac{1}{n} \sum_{i=1}^{\infty}\frac{1}{q_i}\right) p_1}\left\|\left|f_{j}\right| \right\|_{L^{\vec{q}_{1}}} ^{p_{1}}\right\}
\leq C\|f\|^{p_1}_{M\dot{K}_{p_1,\vec{q}_1}^{\alpha,\lambda}}.
$$
That is also why $\alpha<n\left(1-\frac{1}{n}\sum\limits_{i=1}^n\frac{1}{q_i}\right). $
$F_3$ uses the same way, $j>k+2$ , $x\in A_k,$
$$
\begin{aligned}
\left\|M_b^{l} f_{j} \chi_{k}\right\|_{L^{\vec{q}_2}}
 & \leq C2^{-jn/l^{\prime}} \left\| \int_{A_{j}} |b(x)-b(y)| ~|f(y)| dy \chi_{k} \right\|_{L^{\vec{q}_2}}
&\leq C 2^{(k-j) \sum\limits_{i=1}^n\frac{1}{q_{i}}}(j-k)\left\|f_{j}\right\|_{L^{\vec{q}}}.
\end{aligned}
$$
Then
$$
\begin{aligned}
F_1 &\leq C\mathop{sup}\limits_{k_{0}\in \mathbb{Z}} 2^{-k_{0}\lambda p_1} \left\{\sum_{k=-\infty}^{k_{0}}2^{k \alpha p_1} \sum_{j=-\infty}^{k-3} (j-k)^{p_1}2^{(k-j) p_1 \sum_{i=1}^{\infty}\frac{1}{q_i}}\left\|\left|f_{j}\right| \right\|_{L^{\vec{q}_{1}}} ^{p_{1}}\right\}\\
&= C \mathop{sup}\limits_{k_{0}\in \mathbb{Z}} 2^{-k_{0}\lambda p_1} \left\{\sum_{k=-\infty}^{k_{0}} \sum_{j=k+3}^{\infty} 2^{j \alpha p_1} (j-k)^{p_1}2^{(k-j)p_1 \left( \alpha
	 + \sum_{i=1}^{\infty}\frac{1}{q_i}\right) }\left\|\left|f_{j}\right| \right\|_{L^{\vec{q}_{1}}} ^{p_{1}}\right\}.
\end{aligned}
$$
The later proof is similar to the previous one, so it only explains the idea of proof. When $0 <p_1 \leq 1 $, using this inequality $\left(\sum_{i=- \infty}^{\infty} |a_j|\right)^p \leq \sum_{i=- \infty}^{\infty} |a_j|^p$. And the obvious inequality $\left\|f_{j}\right\|_{L^{q}\left(\mathbb{R}^{n}\right)}^{p} \leq 2^{-j \alpha p} \sum_{l=-\infty}^{j} 2^{l \alpha p}\left\|f_{l}\right\|_{L^{q}\left(\mathbb{R}^{n}\right)}^{p}$. We can prove that $F_3\leq C\|f\|^{p_1}_{M\dot{K}_{p_1,\vec{q}_1}^{\alpha,\lambda}}$. When $1<p_1 < \infty$, we divide $F_3$ into two parts,
$$
\begin{aligned}
F_3 &\leq C \mathop{sup}\limits_{k_{0}\in \mathbb{Z}} 2^{-k_{0}\lambda p_1} \left\{\sum_{k=-\infty}^{k_{0}} \sum_{j=k+3}^{k_0} 2^{j \alpha p_1} (j-k)^{p_1}2^{(k-j)p_1 \left( \alpha + \sum_{i=1}^{\infty}\frac{1}{q_i}\right) }\left\|\left|f_{j}\right| \right\|_{L^{\vec{q}_{1}}} ^{p_{1}}\right\}\\
&+ C \mathop{sup}\limits_{k_{0}\in \mathbb{Z}} 2^{-k_{0}\lambda p_1} \left\{\sum_{k=-\infty}^{k_{0}} \sum_{j=k_0 +3}^{\infty} 2^{j \alpha p_1} (j-k)^{p_1}2^{(k-j)p_1 \left( \alpha + \sum_{i=1}^{\infty}\frac{1}{q_i}\right) }\left\|\left|f_{j}\right| \right\|_{L^{\vec{q}_{1}}} ^{p_{1}}\right\}\\
&:=C(G_1+G_2).
\end{aligned}
$$
For $G_1$ and $G_2$, the H$\ddot{o}$lder  inequality and the  inequality $\left\|f_{j}\right\|_{L^{\vec{q}}}^{p} \leq 2^{-j \alpha p} \sum_{l=-\infty}^{j} 2^{l \alpha p}\left\|f_{l}\right\|_{L^{\vec{q}}}^{p}$ are used respectively. The $G_1\leq C\|f\|^{p_1}_{M\dot{K}_{p_1,\vec{q}_1}^{\alpha,\lambda}}$ and $G_2\leq C\|f\|^{p_1}_{M\dot{K}_{p_1,\vec{q}_1}^{\alpha,\lambda}}$ can be  demonstrated. 
In conclusion,
$$
\|M_b^{l}f\|^{p_1}_{M\dot{K}_{p_2,\vec{q}_2}^{\alpha,\lambda}}\leq C\|f\|^{p_1}_{M\dot{K}_{p_1,\vec{q}_1}^{\alpha,\lambda}}.
$$
This completes the proof of the theorem.
\end{proof}
Now the main theorem is provided, which reflects the boundedness of commutators generated by sublinear operators and BMO functions on mixed Herz-Morrey spaces. The proof is similar to the above method, so it is left for the reader to prove.

\begin{theorem}
	Let $0\leq \lambda <\infty$, $0<p< \infty$, $1<\vec{q}< \infty$ and $\lambda-\sum\limits_{i=1}^n \frac{1}{q_i} <\alpha<n\left(1-\frac{1}{n}\sum\limits_{i=1}^n\frac{1}{q_i}\right),$ suppose a sublinear operators $[b,T]$ satisfied that\\
	$(1)$ $b\in BMO(\mathbb{R}^n)$ and  $[b,T]$ is bounded on $L^{\vec{q}}\left(\mathbb{R}^n\right);$\\
	$(2)$ for suitable functions $f$ with $suppf\subseteq A_k$ and $|x|\geq2^{k+1}$ with $k\in \mathbb {Z},$
	\begin{equation}
	\mid Tf(x)\mid\leq C\frac{\|f\|_{L^1}}{|x|^n} ;\label{5.6}
	\end{equation}
	$(3)$ for suitable functions $f$ with $suppf\subseteq A_k$ and $|x|\leq 2^{k-2}$ with $k\in \mathbb {Z},$
	\begin{equation}
	\mid Tf(x)\mid\leq C\frac{\|f\|_{L^1}}{2^{nk}} .\label{5.7}
	\end{equation}
	Then $[b,T]$ is also bounded on $M\dot{K}_{p,\vec{q}}^{\alpha,\lambda}(\mathbb{R}^n)$.
\end{theorem}

\begin{corollary}
	Let $0\leq \lambda <\infty$, $0<p< \infty$, $1<\vec{q}< \infty$ and $\lambda-\sum\limits_{i=1}^n \frac{1}{q_i} <\alpha<n\left(1-\frac{1}{n}\sum\limits_{i=1}^n\frac{1}{q_i}\right)$, $~b\in BMO(\mathbb{R}^n)$.  If sublinear operators T satisfied the following condition\\
	\begin{equation}
	\left|Tf(x)\right|\leq C\int_{\mathbb{R}}\frac{\left|f(x)\right|}{\left|x-y\right|^n}dy,\quad \quad f\in L^1(\mathbb{R}^n)  \mbox{ with compact support}\quad x\notin suppf,\label{10}
	\end{equation}
	And if $[b,T]$ is bounded on $L^{\vec{q}}\left(\mathbb{R}^n\right)$, then $[b,T]$ is also bounded on $M\dot{K}_{p,\vec{q}}^{\alpha,\lambda}(\mathbb{R}^n).$
\end{corollary}

\begin{remark}
	It is worth noting that \eqref{10} is satisfied by many operators studied in harmonic analysis, such as Calder$\acute{o}$n-Zygmund operators, Hardy-Littlewood maximal operators, R.Fefferman’s singular integral operators and Bochner-Riesz means at the critical index and so on.
\end{remark}

\begin{theorem}
	Let $0<l<n$, $0\leq \lambda < \infty$, $0<p_1\leq p_2\leq \infty$, $1<\vec{q_1}<\frac{1}{l}$,   $l=\frac{1}{n}\sum\limits_{i=1}^n\frac{1}{q_{1i}}-\frac{1}{n}\sum\limits_{i=1}^n\frac{1}{q_{2i}}~$ and   $\lambda-\sum\limits_{i=1}^n\frac{1}{q_{2i}}<\alpha<n-\sum\limits_{i=1}^n\frac{1}{q_{1i}},$ suppose a sublinear operators $T_l$ satisfied that\\
	$(1)$ $b\in BMO(\mathbb{R}^n)$ and  $[b,T_l]$ is bounded from $L^{\vec{q_1}}\left(\mathbb{R}^n\right)$ to $L^{\vec{q_2}}\left(\mathbb{R}^n\right) ;$\\
	$(2)$ for suitable functions $f$ with $suppf\subset A_j$ and $|x|\geq2^{j+1}$ with $j\in \mathbb {Z},$
	\begin{equation}
	\mid T_lf(x)\mid\leq C\frac{\|f\|_{L^1}}{|x|^{n-l}} ;\label{5.8}
	\end{equation}
	$(3)$ for suitable functions $f$ with $suppf\subset A_j$ and $|x|\leq 2^{j-2}$ with $j\in \mathbb {Z},$
	\begin{equation}
	\mid T_lf(x)\mid\leq C\frac{\|f\|_{L^1}}{2^{j(n-l)}} .\label{5.9}
	\end{equation}
	Then $[b,T_l]$ is also bounded from $M\dot{K}_{p_1,\vec{q}_1}^{\alpha,\lambda}(\mathbb{R}^n)$ to $M\dot{K}_{p_2,\vec{q}_2}^{\alpha,\lambda}(\mathbb{R}^n).$
\end{theorem}

\begin{corollary}
Let $0<l<n$, $0\leq \lambda < \infty$, $0<p_1\leq p_2\leq \infty$, $1<\vec{q_1}<\frac{1}{l}$, $l=\frac{1}{n}\sum\limits_{i=1}^n\frac{1}{q_{1i}}-\frac{1}{n}\sum\limits_{i=1}^n\frac{1}{q_{2i}}~$and $\lambda-\sum\limits_{i=1}^n\frac{1}{q_{2i}}<\alpha<n-\sum\limits_{i=1}^n\frac{1}{q_{1i}},$ $b\in BMO(\mathbb{R}^n)$. If sublinear operators $T_l$ satisfied the following condition\\
	\begin{equation}
\left|T_lf(x)\right|\leq C\int_{\mathbb{R}}\frac{\left|f(x)\right|}{\left|x-y\right|^{n-l}}dy,\quad \quad f\in L^1(\mathbb{R}^n) \mbox{ with compact support}\quad x\notin suppf.\label{13}
\end{equation}
	And if $[b,T_l]$ is bounded from $L^{\vec{q_1}}\left(\mathbb{R}^n\right)$ to $L^{\vec{q_2}}\left(\mathbb{R}^n\right)$, then $[b,T_l]$ is also bounded from $M\dot{K}_{p_1,\vec{q}_1}^{\alpha,\lambda}(\mathbb{R}^n)$ to $M\dot{K}_{p_2,\vec{q}_2}^{\alpha,\lambda}(\mathbb{R}^n).$
\end{corollary}

\hspace*{-0.6cm}\textbf{\bf Competing interests}\\
The authors declare that they have no competing interests.\\

\hspace*{-0.6cm}\textbf{\bf Funding}\\
The research was supported by Natural Science Foundation of China (Grant Nos. 12061069).\\

\hspace*{-0.6cm}\textbf{\bf Authors contributions}\\
All authors contributed equality and significantly in writing this paper. All authors read and approved the final manuscript.\\

\hspace*{-0.6cm}\textbf{\bf Acknowledgments}\\
The authors would like to express their thanks to the referees for valuable advice regarding previous version of this paper.\\

\hspace*{-0.6cm}\textbf{\bf Authors detaials}\\
Mingwei shi and Jiang Zhou*, moluxiangfeng888@163.com and zhoujiang@xju.edu.cn, College of Mathematics and System Science, Xinjiang University, Urumqi, 830046, P.R China.

\bigskip
\noindent Mingwei shi and Jiang Zhou\\
\medskip
\noindent
College of Mathematics and System Sciences\\
Xinjiang University\\
Urumqi 830046\\
\smallskip
\noindent{D-mail }:\\
\texttt{moluxiangfeng888@163.com} (Mingwei shi)\\
\texttt{zhoujiang@xju.edu.cn} (Jiang Zhou)
\bigskip \medskip
\end{document}